\documentclass[12pt]{amsart}
\usepackage{geometry} 
\usepackage{ulem}
\usepackage{pgfplots}
\usepackage{psfrag}
\usepackage{amsmath}
\usepackage{calc}
\usepackage{systeme}
\usepackage{amsfonts}
\usepackage{amsthm}
\usepackage{algorithm2e}
\usepackage[applemac]{inputenc}
\usepackage{bbold}
\usepackage[numbers]{natbib}
\newtheorem{hyp}{Hypothesis}

\numberwithin{equation}{section}

\author{%
  Charles Bertucci $^1$  }

\newtheorem{Theorem}{Theorem}[section]
\newtheorem{Lemma}{Lemma}[section]
\newtheorem{Rem}{Remark}[section]
\newtheorem{Def}{Definition}[section]
\newtheorem{Prop}{Proposition}[section]

\usepackage[foot]{amsaddr}

\newcommand{\be}{\begin{equation}}
\newcommand{\ee}{\end{equation}}
\newcommand{\R}{\mathbb{R}}
\newcommand{\T}{\mathbb{T}}
\newcommand{\eps}{\epsilon}
\newcommand{\mone}{\mathcal{M}_1(\T^d)}

\newcommand{\dmU}{\frac{\delta U}{\delta m}}
\newcommand{\mptd}{\mathcal{P}(\mathbb{T}^d)}
\newcommand{\mmtd}{\mathcal{M}(\mathbb{T}^d)}

\title{Monotone solutions for mean field games master equations: continuous state space and common noise}
\thanks{$^1$ : CMAP, Ecole Polytechnique, UMR 7641, 91120 Palaiseau, France
}
\date{} 

\begin{document}
\maketitle
\begin{abstract}
We present the notion of monotone solution of mean field games master equations in the case of a continuous state space. We establish the existence, uniqueness and stability of such solutions under standard assumptions. This notion allows us to work with solutions which are merely continuous in the measure argument, in the case of first order master equations. We study several structures of common noises, in particular ones in which common jumps (or aggregate shocks) can happen randomly, and ones in which the correlation of randomness is carried by an additional parameter.\end{abstract}
\tableofcontents
\section*{Introduction}

This paper introduces the notion of monotone solution for mean field games (MFG in short) master equations in the case of a continuous state space. Using this notion, we establish results of existence and uniqueness for merely continuous solutions of master equations, which are non-linear first order infinite dimensional partial differential equations (PDE in short) at the core of the MFG theory. Even though this paper is self-contained, it is the follow-up of \citep{bertucci2020monotone} in which we presented a similar notion in the simpler case of a finite state space. 

\subsection*{General introduction}
MFG are dynamic games involving a crowd of non-atomic agents. If such games have a tremendous number of applications in several fields, they naturally arise in Economics, and they actually did so in the eighties and nineties. A general mathematical framework to study such games (as well as the terminology MFG) has been introduced by J.-M. Lasry and P.-L. Lions in \citep{lasry2007mean,lions2007cours}. We present here some general aspects of this theory, focusing on Nash Equilibria and the concept of value of such games. The study of the value of a MFG reduces to the analysis of a PDE called the master equation \citep{lions2007cours,cardaliaguet2019master}. The fact that a value can be defined is a consequence of strong uniqueness properties of Nash equilibria in the so-called monotone regime, in which, formally, players are adversarial to one another. The master equation is generally an infinite dimensional non linear PDE as soon as the state space of the players is continuous. Another striking property of MFG is that when the randomness (or noise) affecting the players is distributed in an i.i.d. fashion among them, Nash equilibria can be characterized with a system of forward-backward PDE in finite dimension \citep{lasry2007mean,lions2007cours,huang2006large,porretta2015weak}. Let us note that, in this situation, several Nash equilibria can coexist. An important aspect of the MFG theory is the so-called probabilistic approach \citep{carmona2018probabilistic,lacker2016general}, which we shall not particularly use in this paper. Finally, we end this general introduction by mentioning the question, that we do not treat here, of the convergence of $N$-players games toward MFG, which is a possible way to justify the PDE arising in MFG theory and which has been partially solved at this point  \citep{cardaliaguet2019master,lacker2018convergence}.
 
\subsection*{Regularity of the solution of the master equation}
In the aforementioned monotone regime, the uniqueness result established by Lasry and Lions \citep{lasry2007mean} makes it meaningful to define a value function associated to a MFG. This value function associates the value of the game $U(t,x,m)$ to a player in a state $x$, when the remaining time in the game is $t$ and the measure $m$ describes the repartition of the other players. In this monotone regime, if it is smooth, the value function can be characterized as the unique solution of the master equation \citep{lions2007cours,cardaliaguet2019master}. A natural and fundamental question in the MFG theory is the following : If the value function is not smooth, can it still be characterized as the unique weak solution of the master equation ? The difficulty here lies in the definition of weak solutions one has to choose. In this paper, we answer positively to this question in the monotone regime. Namely, to define our monotone solutions, we only need continuity of the value function in the measure argument for first order master equations and first order regularity with respect to the measure argument for second order master equations. This paper is the extension of \citep{bertucci2020monotone} in which we treated this problem in the case of a finite state space. Let us insist on the fact that, outside the monotone regime, the concept of a value of a MFG does not make sense because of the multiplicity of equilibria. If mathematical questions can still be asked on the master equation in a general regime, their link with the underlying game and their practical interest is doubtful.

In the last years, the question of the regularity of solutions of master equations has raised quite a lot of interest. In \citep{cardaliaguet2019master}, the authors shew that, under strong assumptions, the value function turns out to be smooth and, thus, the unique solution of the master equation. Alternatively, the monotone regime proved to be regularizing in the finite state space case \citep{lions2007cours,bertucci2019some,bertucci2021master}. More recently, several teams have addressed the issue of defining weak solutions of master equations in several context (which are not the monotone regime): \citep{cecchin2019convergence,cecchin2022selection} propose ways of selecting a weak solution in finite state space, particularly in the potential case; \citep{mou,gangbo,zhang} introduced other notions of weak solutions of the master equation. In \citep{mou}, the authors gave a characterization of the value function which relies on the stability of the value, due to the monotone regime. Such an idea is also present in this paper, even if we are here able to characterize the value function without using the system of characteristics. In \citep{gangbo} the authors are concerned with the potential case, in which the master equation is reduced to a Hamilton-Jacobi equation in infinite dimension. In \citep{zhang}, the authors studied the master equation under other geometric assumptions than the monotone one (which they called displacement monotonicity by reference to displacement convexity).  Let us also mention the paper \citep{zhang2} which is concerned with the study of the master equation in a non-monotone regime. Up to this point, no general framework has been proposed and we think that this paper can start one, in the monotone regime.

\subsection*{Modeling of common noise}
An objective of this paper is to introduce a notion of weak solution for the MFG master equations (in the monotone regime). As already mentioned, in the absence of a common noise, the study of the master equation is not necessary. Hence, it is natural to present our notion of solution in cases involving a common noise. Up to now, most of the mathematical literature on MFG \citep{cardaliaguet2019master,carmona2018probabilistic,carmona2016common,cardaliaguet2022first} is concerned with the following common noise: the state of all players is affected by a common Brownian motion. This noise has several specificities: it yields second order terms (with respect to the measure variable) in the master equation; it is not particularly regularizing (at least this has not been established up to now and it is quite unlikely since the arising terms are only "degenerate elliptic"); it induces a singular behavior for the underlying measure which describes the repartition of players in the state space. Indeed in this context, this measure is randomly pushed with a force which is a Brownian motion. If this type of noise is certainly helpful to model numerous situations, we here argue that it is undeniably not the most general situation and that several other models can be of interest in many applications, and that those models do not raise mathematical questions as difficult as the one we just mentioned. Let us precise that the second to last Section of this paper is concerned with this often-studied common noise.\\

A first type of common noise we insist on later on in the paper, is a setting in which the correlation in the randomness affecting the players is carried by an additional parameter. Typical examples for such kinds of models are MFG involving players which interact on a market through stochastic environmental variables, such as prices for instance. In such a context, it is natural to expect that the value of the game depends on this additional parameter, and therefore that the associated master equation depends on derivatives (possibly second order ones) of the value function with respect to these additional parameters. Clearly, if this additional parameter is finite dimensional, then the master equation stays a first order PDE (with respect to the measure variable) despite modeling a MFG with common noise.\\

A second type of noise we want to model is one similar to the common noise introduced in \citep{bertucci2019some} in the case of a finite state space. This type of noise consists in assuming that at random times, which are common to all players, a transformation is going to affect all players in the game. Such a type of noise is adequate to model aggregate shocks (to use a terminology from Economics) which may occur at random times. In addition to its intrinsic (mathematical) interest, this type of noise is helpful to approximate other types of common noises, such as the one in which all the players are pushed by the same Brownian motion.

\subsection*{Main results of the paper}
The major part of this paper is concerned with defining notions of monotone solutions and establishing results of uniqueness of such solutions in various contexts. We do so for stationary and time dependent master equations, without common noise, in Section \ref{sec:mon}. We extend this notion to situations involving common noise in three different settings in Section \ref{sec:firstcom}, all of which leads to master equation of first order. In all the previous cases, the value function is only assumed to be continuous with respect to the measure argument. We then partially treat the case of master equations with common noise of second order in Section \ref{sec:secondorder}. We say partially because in this context, the value function is required to be differentiable with respect to the measure argument. This still consists in an improvement with respect to the results of \citep{cardaliaguet2019master}, but it could be possible to extend this notion to merely continuous functions (in the measure argument).

We also provide a priori estimates concerning the regularity of the solution of the master equation, namely we prove Propositions \ref{holderU} and \ref{prop:est}. The former is an extension of a result of \citep{cardaliaguet2019master} in a case without common noise, while the latter is concerned with master equations involving a common noise. We use those two a priori estimates to prove three results of existence of monotone solutions. The first one is Theorem \ref{existmon}. It states the existence of monotone solutions of master equations without common noise, and it does not state more than H\"older regularity with respect to the measure argument for the value function. The second one is Theorem \ref{thm:existt} and it is concerned with the existence of monotone solutions of a first order master equation involving common noise. The last result of existence is about monotone solutions for second order master equations and it is given in Theorem \ref{thm:exist2}. All those results are clear improvements from the existing results of existence of solutions of master equations, as they require quite less restrictive smoothness assumptions on the datas of the problem, like it is the case in \citep{cardaliaguet2019master} for instance.

Additionally, we also make the link between common noise which yields first order terms in the master equation and common noise which leads to second order terms in the master equation. This is done in Section \ref{sec:translations}.

\subsection*{Structure of the paper}
In Section \ref{sec:notation}, we introduce some notation as well as recall some results concerning derivatives in the space of probability measure. In Section \ref{sec:model}, we present the main MFG model underlying the master equations we study, as well as some known results concerning MFG master equations. We proceed by introducing our notion of monotone solution in Section \ref{sec:mon}. We discuss the question of the existence of monotone solutions in Section \ref{sec:exist}. Sections \ref{sec:firstcom} and \ref{sec:secondorder} deal with respectively monotone solutions of master equations of first and second order, in the presence of a common noise. We conclude this paper and present perspectives and extensions of this work in Section \ref{sec:conclusion}.

\section{Notation and derivatives in the space of probability measure}\label{sec:notation}
In this somehow introductive section, we present some notations, especially concerning derivatives in a set of probability measures, as well as some basic results on those derivatives that we shall need later on.
\begin{itemize}
\item We denote by $d$ an integer greater than $1$ which refers to the dimension of the players' state space.
\item We denote by $\mathbb{T}^d$ the $d$ dimensional torus whose inner scalar product is denoted by $\cdot$, i.e. $x\cdot y$ denotes the scalar product between $x,y \in \mathbb{T}^d$.
\item The set of measures on $\mathbb{T}^d$ is denoted by $\mmtd$. For $m \in \mmtd$, we denote its support by $Supp(m)$.
\item We denote by $\mptd$ the set of probability measures on $\mathbb{T}^d$. This set is equipped with the Monge-Kantorovich (or $1$-Wasserstein) distance $\textbf{d}_1$ defined with
\begin{equation}
\forall µ,\nu \in \mptd, \textbf{d}_1(µ,\nu) = \sup_{\phi} \int_{\mathbb{T}^d}\phi(x) (µ-\nu)(dx),
\end{equation}
where the supremum is taken over all Lipschitz functions $\phi :\mathbb{T}^d \to \mathbb{R}$ with a Lipschitz constant at most one. We recall that $\left(\mptd,\textbf{d}_1\right)$ is a compact metric set.
\item We denote by $\langle \cdot,\cdot\rangle$ the scalar product of $L^2(\mathbb{T}^d)$ and by a slight abuse of notation, all its extensions on functional spaces in duality. That is, if $f$ and $µ$ are in $L^2(\mathbb{T}^d)$, then $\langle f,µ\rangle$ is their $L^2(\mathbb{T}^d)$ scalar product, but if $f \in \mathcal{C}^0(\mathbb{T}^d)$ and $µ \in \mptd$, then $\langle f,µ\rangle$ is the integral of $f$ against the measure $µ$ ; and if for instance $f \in \mathcal{C}^2(\mathbb{T}^d)$ and $µ \in \mptd$, then $\langle f, \Delta µ\rangle$ is the evaluation of the distribution of second order $\Delta µ$ on $f$.
\item For a function of two variables $f : (\mathbb{T}^d)^2 \to \mathbb{R}$, we define, whenever it makes sense
\begin{equation}
\langle µ | f(\cdot,\cdot)|\nu\rangle = \langle f_1, \nu\rangle,
\end{equation}
where $f_1 : y \to \langle f(\cdot,y),µ \rangle$. The important remark is that $µ$ is tested against the first argument of $f$ whereas $\nu$ is tested against its second argument.
\item An application $f : \mptd \to \mathcal{C}^0(\mathbb{T}^d)$ is said to be monotone if 
\begin{equation}
\forall µ,\nu \in \mptd, \langle f(µ)(\cdot) - f(\nu)(\cdot), µ - \nu \rangle \geq 0.
\end{equation}
\item For $n \in \mathbb{N}, \alpha \in [0,1)$ and a function $\phi : \mathbb{T}^d \to \mathbb{R}$ we denote by $\|\phi\|_{n+\alpha}$ its $\mathcal{C}^{n,\alpha}$ norm.
\item For $n \in \mathbb{N}, \alpha \in [0,1)$ and a function $\phi : \mathbb{T}^d\times \mathbb{T}^d \to \mathbb{R}$ we denote by $\|\phi\|_{(n+ \alpha,n+\alpha)}$ its $\mathcal{C}^{n,\alpha}$ norm.
\item We introduce the space $\mathcal{B}$ of functions $U : \mathbb{T}^d \times \mptd \to \mathbb{R}$ such that $U$ is globally continuous and
\begin{equation}
\sup_{m \in \mptd} \left\| U(\cdot,m)\right\|_{2} < \infty.
\end{equation}
We also note $\mathcal{B}^c$ the subspace of $\mathcal{B}$ of functions which are continuous seen as $\mptd \to \mathcal{C}^2$.
\item We introduce the space $\mathcal{B}_t$ of functions $U : [0,\infty) \times \mathbb{T}^d \times \mptd \to \mathbb{R}$ such that $U$ is globally continuous and for all $T > 0$ :
\begin{equation}
\sup_{m \in \mptd, t\in [0,T]} \left\| U(t,\cdot,m)\right\|_{2} < \infty.
\end{equation}
We also note $\mathcal{B}_t^c$ the subspace of $\mathcal{B}_t$ of functions which are continuous seen as $[0,\infty)\times \mptd \to \mathcal{C}^2$.
\item  We introduce the space $\mathcal{B}'_t$ of functions $U : [0,\infty) \times \mathbb{T}^d \times \mptd \times \mathbb{T}^k\to \mathbb{R}$ such that $U$ is globally continuous and for all $T > 0$ :
\begin{equation}
\sup_{m \in \mptd, t\in [0,T]} \left\| U(t,\cdot,m,\cdot)\right\|_{2} < \infty.
\end{equation}
\item The usual convolution product in $\mathbb{T}^d$ is denoted by $\star$.
\item The image measure of a measure $m$ by a map $T$ is denoted $T_{\#}m$.

\end{itemize}
\subsection{Derivatives in the space of probability measures}
Recall that $\mptd$ is equipped with $\textbf{d}_1$. We say that a function $U : \mptd \to \mathbb{R}$ is derivable at $m$ if there exists a continuous map $\phi : \mathbb{T}^d \to \mathbb{R}$ such that for all $ µ \in \mptd$
\begin{equation}\label{defder}
\lim_{\theta \to 0, \theta > 0} \frac{U((1 - \theta)m + \theta µ) - U(m)}{\theta} = \langle \phi, µ - m\rangle.
\end{equation}
Clearly, there is no uniqueness of such a function $\phi$ as it is defined up to a constant. We denote $\phi = \dmU(m)$ when it is such that $\langle \phi,m\rangle = 0$.\\

We say that $U$ is $\mathcal{C}^1$ if the map $m \to \dmU(m)$ is continuous.\\

 When $\dmU(m)$ is a $\mathcal{C}^1$ function of $\mathbb{T}^d$, we denote its gradient by $D_m U(m,x) := \nabla_x \dmU(m)(x)$. The function $D_mU$ is the intrinsic derivative of $U$ at $m$. It satisfies
 \begin{equation}\label{defintrinsic}
 \lim_{h \to 0}\frac{U((Id + h \phi)_{\#}m) - U(m)}{h} = \left \langle D_mU(m)\cdot\phi,m\right\rangle,
 \end{equation}
where $\phi:\mathbb{T}^d \to \mathbb{R}^d$ is a continuous function.\\

If it exists, the second order derivatives of $U : \mptd \to \mathbb{R}$ at $m$ is a function $\psi : (\mathbb{T}^d)^2 \to \mathbb{R}$ such that for any $x \in \mathbb{T}^d$, $\psi(x,\cdot ) = \frac{\delta}{\delta m} \left( \dmU(\cdot,x)\right)(\cdot)$. We denote by $\frac{\delta^2 U}{\delta m^2}$ the map $\psi$ such that for any $x \in \mathbb{T}^d, \left\langle \psi,m\right \rangle = 0$.\\

Let us finally introduce the following norm on functions $U : \mathbb{T}^d \times \mptd \to \mathbb{R}$ for $n \in \mathbb{N}$ and $\alpha \in [0,1)$
\begin{equation}
\text{Lip}_{n+\alpha}(U) = \sup_{µ,\nu \in \mptd} \frac{\|U(µ) - U(\nu)\|_{n+\alpha}}{\textbf{d}_1(µ,\nu)}.
\end{equation}
We define in the same way for $U : (\mathbb{T}^d)^2 \times \mptd \to \mathbb{R}$
\begin{equation}
\text{Lip}_{n+\alpha}(U) = \sup_{µ,\nu \in \mptd} \frac{\|U(µ) - U(\nu)\|_{(n+\alpha,n+\alpha)}}{\textbf{d}_1(µ,\nu)}.
\end{equation}

\subsection{First order conditions in $\mptd$}
Consider a $\mathcal{C}^{1}$ function $U : \mptd \to \mathbb{R}$ and $m_0 \in \mptd$ such that 
\begin{equation}\label{opt}
U(m_0) = \inf_{\mptd} U(m).
\end{equation}
One would like to have that $\dmU(m_0) = 0$, however this is not true in general. This is mainly due to the fact that , formally, $\mptd$ has many boundaries and that the optimality conditions associated to \eqref{opt} only yield an inequality in general. We can establish the following.
\begin{Prop}\label{firstcond}
Let $U$ be a $\mathcal{C}^{1}$ function on $\mptd$ which attains its minimum at $m_0$. Then $0 \leq\dmU(m_0)$ attains its minimum $0$ on $Supp(m_0)$. Moreover, if $\dmU(m_0)$ is $C^{1}$, then for any smooth $\phi : \mathbb{T}^d \to \mathbb{R}^d$
\begin{equation}\label{optdm}
\left \langle \phi \cdot D_m U(m_0),m_0\right \rangle = -\left \langle \dmU(m_0),\text{div}(\phi m_0)\right \rangle = 0,
\end{equation}
if $\dmU(m_0)$ is $C^2$, then 
\begin{equation}\label{optDelta}
\left \langle \dmU(m_0),\Delta m_0\right \rangle \geq  0,
\end{equation}
\end{Prop}
\begin{proof}
Let $x \in \mathbb{T}^d$ and, recalling the definition of $\dmU$ \eqref{defder} for $µ = \delta_x$, we obtain that $\dmU(m_0,x) \geq 0$. Because we have the normalization condition $\langle \dmU(m_0),m_0\rangle = 0$, we deduce that $U$ reaches its minimum $0$ on $Supp(m_0)$. The rest of the claim follows quite easily from the optimality conditions in $\mathbb{T}^d$.
\end{proof}

One could also provide general results for second order conditions in the spirit of what we just did. Such results are not presented because they are of no need in the following.

\section{Main model and preliminaries}\label{sec:model}

In this section we present the typical master equations we are going to study as well as the underlying MFG model. We also give the main assumptions for the rest of the paper and recall an existing result of uniqueness and a variation of a Lemma from Stegall.
\subsection{Mean Field Games and master equations}
We recall, on a well known example, the links between MFG and master equations. We assume that a crowd of non-atomic agents evolves in $\mathbb{T}^d$ during the time interval $[0,t_f]$. The state $(X_t)_{t \geq 0}$ of a player follows
\begin{equation}
dX_t = \alpha_t dt + \sqrt{2 \sigma}dW_t,
\end{equation}
where $(\alpha_t)_{t \geq 0}$ is the control of the player and $(W_t)_{t \geq 0}$ is a standard $d$ dimensional Brownian motion on $\mathbb{T}^d$ which models individual noise. By individual noise, we mean that two players' states are going to evolve according to the previous stochastic differential equation for two independent realizations of $(W_t)_{t \geq 0}$. We assume that the cost of a player whose state and control are given by $(X_t)_{t \geq 0}$ and $(\alpha_t)_{t \geq 0}$ is given by
\begin{equation}
\int_0^{t_f} L(X_t,\alpha_t) + f(X_t,m_t) dt + U_0(X_{t_f},m_{t_f}),
\end{equation}
where $L$, $f$ and $U_0$ are cost functions and $(m_t)_{t \geq 0}$ is the process which describes the evolution of the measure describing the repartition of the players in the state space. Hence $L$ represents the part of the cost the player pays which depends on its control, whereas $f$ is the part which depends on the other players. The function $U_0$ represents a final cost.
\begin{Rem}
In this paper we work only on the so-called decoupled case, in which the dependence on $\alpha$ and $m$ are separated in the previous equation. The extension of the following study to coupled cases is not necessary trivial since coupled hamiltonians which satisfy appropriate monotonicity assumptions tend to be singular in measure space.
\end{Rem}

Denoting by $U(t,x,m)$ the value of the game (which is not clearly defined at this moment) for a player in the state $x$ when it remains $t$ time in the game and the distribution of the other players in the state space is currently $m$, we obtain that $U$ solves (if it smooth enough) the so-called master equation 
\begin{equation}\label{me}
\begin{aligned}
\partial_t U& - \sigma \Delta U + H(x,\nabla_x U) - \left\langle \dmU(x,m,\cdot), \text{div}\left(D_pH(\cdot, \nabla_x U(\cdot,m))m\right)\right\rangle\\
& - \sigma\left\langle \dmU(x,m,\cdot), \Delta m \right\rangle = f(x,m), \text{ in } (0,\infty)\times \mathbb{T}^d\times \mptd\\
U(0,&x,m) = U_0(x,m) \text{ in }\mathbb{T}^d\times \mptd,
\end{aligned}
\end{equation}
where $H$ is the Hamiltonian of the players given by $H(x,p) := \sup_{\alpha}\{-\alpha\cdot p - L(x,\alpha)\}$.\\

In the present case, because the noise is only distributed in an i.i.d. fashion among the players, we can characterize Nash equilibria of the game which lasts a time $t_f$ and starts with an initial distribution of player given by $m_0$, using the following system of finite dimensional PDE
\begin{equation}\label{MFGsystem}
\begin{cases}
- \partial_t u - \sigma \Delta u + H(x, \nabla_x u) = f(x,m) \text{ in } (0,t_f)\times \mathbb{T}^d,\\
\partial_t m - \sigma \Delta m - \text{div}(D_pH(x, \nabla_x u)m) = 0 \text{ in } (0,t_f)\times \mathbb{T}^d,\\
u(t_f,x) = U_0(x, m(t_f)), m(0,x) = m_0(x) \text{ in } \mathbb{T}^d.
\end{cases}
\end{equation}
In the previous system, a solution $(u,m)$ is associated to a Nash equilibria of the game in the following way. The distribution of players $m$ evolves according to the second equation of \eqref{MFGsystem} and under the anticipation that the distribution of players is indeed going to be $m$, the value of the game for the players is given by $u$. A particular set of MFG are the one called monotone, i.e. for which the following assumption is satisfied.
\begin{hyp}\label{hypmon}
The Hamiltonian $H$ is convex in its second argument. The couplings $f$ and $U_0$ are monotone, i.e. they verify for all $µ,\nu \in \mptd$
\begin{equation}\label{monf}
\langle f(\cdot,µ) - f(\cdot,\nu), µ - \nu \rangle \geq 0,
\end{equation}
\begin{equation}\label{monu0}
\langle U_0(\cdot,µ) - U_0(\cdot,\nu), µ - \nu \rangle \geq 0.
\end{equation}
\end{hyp}
If the previous assumption is satisfied, and that for either \eqref{monf} or \eqref{monu0} the inequality is strict as soon as $µ \ne \nu$, then there exists at most one solution of \eqref{MFGsystem} for any initial condition $m_0$ and any duration of the game $t_f$. Hence we deduce from this strong uniqueness result for Nash equilibria of the MFG, that a concept of value of a game can be defined. By this we mean that we can indeed talk about the value $U(t,x,m)$ of the MGF for a player in the state $x$, when the time remaining in the game is $t$, and the repartition of players in the state space is described by $m$. In this context, the value $U$ obviously satisfies for all $t_f \geq 0, x \in \mathbb{T}^d, m_0 \in \mptd, U(t_f,x,m_0) = u(0,x)$, where $u$ is such that $(u,m)$ is the unique solution of \eqref{MFGsystem}. Clearly if $U$, defined in this way, is smooth, then it is a solution of \eqref{me}.\\

One of the main objectives of this paper is to generalize the previous approach to a situation in which the use of a system of characteristics such as \eqref{MFGsystem} is not clear, for instance in the presence of common noise (i.e. a noise which is not distributed in an i.i.d. fashion among the players). Mainly, we are going to establish that we can characterize, under Hypothesis \ref{hypmon}, a value function $U$ for the MFG as the sole weak solution of the master equation, without needing derivability of $U$ with respect to the measure argument.\\

Even though we do not detail the model underlying the following stationary counterpart of \eqref{me}, it could have been presented in the same manner. 
\begin{equation}\label{sme}
\begin{aligned}
rU &- \sigma \Delta U + H(x,\nabla_x U) - \left\langle \dmU(x,m,\cdot), \text{div}\left(D_pH(\cdot, \nabla_x U(\cdot,m))m\right)\right\rangle \\
&- \sigma\left\langle \dmU(x,m,\cdot), \Delta m \right\rangle = f(x,m) \text{ in }\mathbb{T}^d\times \mptd.
\end{aligned}
\end{equation}
This stationary master equation is also a subject of study for this paper.
\begin{Rem}
In the rest of the paper, the presence of the i.i.d. noise between the players plays a crucial role in our study. In the case in which such a noise is not present, let us mention what seems to be the most natural way to formulate the master equation. It is the so-called Hilbertian approach introduced by P.-L. Lions in \citep{lions2007cours}. In this context, the master equation is posed on an Hilbert space and the problem is closer to the finite dimensional setting introduced in \citep{bertucci2020monotone}.
\end{Rem}

\subsection{Preliminary results}
In this section we recall the two main results of existence and uniqueness on master equations which we can find in \citep{cardaliaguet2019master}, as well as a variant of a Lemma of Stegall on approximated optimization.\\

The following Theorem of existence of classical solutions is borrowed from \citep{cardaliaguet2019master}. We do not reproduce its rather long proof, but let us mention that it relies on a precise study of the system \eqref{MFGsystem} and its dependence on the initial conditions. In some sense, a contribution of this paper is to provide another existence result for master equations, which relies on weaker assumptions.
\begin{Theorem}\label{existclassic}
Assume that there exists $C > 0, \alpha \in (0,1)$ such that :
\begin{itemize}
\item The Hamiltonian $H$ satisfies 
\begin{equation}\label{hypH}
\forall x \in \mathbb{T}^d, p \in \mathbb{R}^d, 0 < D^2_{pp} H(x,p) \leq C Id,
\end{equation}
in the sense of symmetric matrices.
\item \begin{equation}
\sup_{m \in \mptd} \left(   \|f(\cdot,m)\|_{2 + \alpha} + \left\| \frac{\delta f(\cdot,m,\cdot)}{\delta m}  \right\|_{(2+ \alpha, 2 + \alpha)}  \right) + \text{Lip}_{2 +\alpha} \left( \frac{\delta f}{\delta m} \right) \leq C.
\end{equation}
\item \begin{equation}
\sup_{m \in \mptd} \left(   \|U_0(\cdot,m)\|_{3 + \alpha} + \left\| \frac{\delta U_0(\cdot,m,\cdot)}{\delta m}  \right\|_{(3+ \alpha, 3 + \alpha)}  \right) + \text{Lip}_{3 +\alpha} \left( \frac{\delta U_0}{\delta m} \right) \leq C.
\end{equation}
\end{itemize}
Then there exists a classical solution $U, \mathcal{C}^1$ in all the variables, $\mathcal{C}^2$ in the space variable $x$, of the master equation \eqref{me}.
\end{Theorem}
The next result is concerned with uniqueness of solutions of master equations. The uniqueness of a smooth solution is a somehow classical (or at least expected) result on PDE. Indeed, uniqueness of smooth solutions of a non linear PDE can often be established without particular structural assumptions. Even if the next result is concerned with such a property, we present a proof which : i) is rarely given in this form, ii) is at the core of the definition of monotone solutions that we give in the next section, iii)relies somehow weakly on the regularity of the solutions.
\begin{Prop}\label{uniqsmooth}
Under Hypothesis \ref{hypmon} : i) there exists at most one smooth solution $U$ of \eqref{me}, moreover if it exists, $U(t)$ is monotone for all $t\geq 0$ ; ii) there exists at most one classical solution of \eqref{sme} and if it exists, it is monotone.
\end{Prop}
\begin{proof}
We only detail the proof of the uniqueness property for the stationary equation, the time dependent being treated using similar arguments. Moreover, we prove more general results in the next section.\\

Let $U$ and $V$ be two smooth solutions of the master equation \eqref{sme}. Let us define the function $W$ on $\mathcal{P}(\mathbb{T}^d)^2$ by
\begin{equation}
W(µ,\nu) = \langle U(\cdot,µ) - V(\cdot, \nu), µ-\nu\rangle := \int_{\mathbb{T}^d} U(x,µ) - V(x,\nu) (µ- \nu)(dx).
\end{equation}
Using the equations satisfied by both $U$ and $V$, we deduce that $W$ satisfies on $\mptd^2$
\begin{equation}
\begin{aligned}
r&W + \langle H(x, \nabla_x U) - H(x,\nabla_x V),µ - \nu \rangle - \left \langle \frac{\delta W}{\delta µ}, \text{div}\left(D_pH(\cdot, \nabla U(\cdot,µ))µ\right)\right\rangle \\
-& \sigma \left \langle \frac{\delta W }{\delta µ}, \Delta µ\right \rangle - \sigma \left \langle \frac{\delta W }{\delta \nu}, \Delta \nu\right \rangle + \langle U - V, \text{div}\left(D_pH(\cdot, \nabla U(\cdot,µ))µ - D_pH(\cdot, \nabla V(\cdot,\nu))\nu\right) \rangle\\
-& \left \langle \frac{\delta W}{\delta \nu}, \text{div}\left(D_pH(\cdot, \nabla V(\cdot,\nu))\nu\right)\right\rangle = \langle f(\cdot,µ) - f(\cdot,\nu),µ- \nu\rangle.
\end{aligned}
\end{equation}
To establish the previous equation, we have used the relations (which are true up to a function $c : \mptd \to \mathbb{R}$)
\begin{equation}\label{reluw}
U(x,µ) - V(x,\nu) + \left\langle \dmU(\cdot,µ,x),µ - \nu\right\rangle = \frac{\delta W}{\delta µ}(µ,\nu,x)  \text{ for all } x \in \mathbb{T}^d, µ, \nu \in \mptd,
\end{equation}
\begin{equation}
V(x,\nu) - U(x,µ) + \left\langle \frac{\delta V}{\delta m}(\cdot,\nu,x),\nu - \mu\right\rangle = \frac{\delta W}{\delta \nu}(µ,\nu,x) \text{ for all } x \in \mathbb{T}^d, µ, \nu \in \mptd.
\end{equation}

The continuous function $W$ reaches its minimum at some point $(µ^*,\nu^*)$ at which the following holds
\begin{equation}\label{inter}
\begin{aligned}
r W(µ^*, \nu^*) &+ \langle H(\cdot,\nabla_x U) - H(x,\nabla_x V), µ^* - \nu^*\rangle - \langle \nabla_x (U- V) \cdot \nabla_x D_p H(\nabla_x U),µ^*\rangle\\
&- \langle \nabla_x (V- U) \cdot \nabla_x D_p H(\nabla_x V),\nu^*\rangle \geq \langle f(\cdot,µ^*) - f(\cdot,\nu^*),µ^* - \nu^*\rangle,
\end{aligned}
\end{equation}
where we have used the optimality conditions given by Proposition \ref{firstcond}. Hence we deduce from Hypothesis \ref{hypmon} that $rW(µ^*,\nu^*) \geq 0$ and thus that $W$ is a non-negative function. From the non-negativity of $W$, we obtain that for all $m \in \mptd, x \in \mathbb{T}^d$ : 
\begin{equation}
U(x,m) = V(x,m) + c(m)
\end{equation}
 for a function $c : \mptd \to \mathbb{R}$. Indeed otherwise $W$ should change sign around the diagonal of $\mathcal{P}(\mathbb{T})^2$. This is easily observed by evaluating, for $x \in \mathbb{T}^d$, $0< \epsilon < 1$
\begin{equation}
W(m,(1-\epsilon) m + \epsilon \delta_x) = \epsilon \langle U(\cdot,m) - V(\cdot,(1-\epsilon)m + \epsilon \delta_x), m - \delta_x\rangle.
\end{equation}
Indeed, the previous yields, by taking the limit $\epsilon \to 0$ and using the fact that $W \geq 0$,
\begin{equation}
c(m) := \langle U(\cdot,m)-V(\cdot,m),m\rangle \geq U(x,m) - V(x,m).
\end{equation}
Looking at $W((1-\epsilon) m + \epsilon \delta_x,m)$ gives the reverse inequality.\\

Evaluating \eqref{sme} for $U$ and $V$ immediately implies that $c(m) = 0$ on $\mptd$. Since $U = V$, we finally obtain that $U$ is monotone from the non-negativity of $W$.\\
\end{proof}

We end these preliminary results with the following variation of Stegall variational principle. Although this extension seems to be new, it is a rather immediate adaptation of existing results the interested could find in the monologue \citep{phelps} for instance.
\begin{Lemma}\label{Stegall}
Let $f : \mptd \to \mathbb{R}$ be a continuous function. Take $m \in \mathbb{N}, m > 6d$. For any $\epsilon > 0$, there exists $\phi$ in the Sobolev space $H^{m}(\mathbb{T}^d), \|\phi\|_{H^m} \leq \epsilon$ such that 
$µ \to f(µ) + \langle \phi,µ\rangle$ has a strict minimum on $\mptd$.
\end{Lemma}
\begin{proof}
Let us consider the multivalued operator $A : H^{m}(\mathbb{T}^d) \substack{\rightarrow \\[-1em] \rightarrow} \mptd$ which is defined by $A(\phi) = \text{argmin} \{f(µ) + \langle \phi,µ\rangle | µ \in \mptd\}$. By construction $-A$ is cyclically monotone in the sense that, for finite sequences $\phi_0, \phi_1, ..., \phi_n = \phi_0$, $µ_i \in A(\phi_i)$, 
\begin{equation}
\sum_{i=1}^n \langle \phi_i - \phi_{i-1}, µ_i\rangle \leq 0.
\end{equation}
Indeed for such a sequence,
\begin{equation}
\begin{aligned}
\sum_{i=1}^n \langle \phi_i - \phi_{i-1}, µ_i\rangle &= \sum_{i=1}^n \langle \phi_i , µ_i - µ_{i+1}\rangle\\
& \leq \sum_{i =1 }^n f(µ_{i+1}) - f(µ_i)\\
&= 0.
\end{aligned}
\end{equation}
We can take $µ_0 \in A(0)$ and construct a function $\psi : H^m(\mathbb{T}^d) \to \mathbb{R}$ by setting
\begin{equation}
\psi(\phi) = \sup \{ \langle \phi - \phi_n, µ_n\rangle + \langle \phi_n - \phi_{n-1},µ_{n-1}\rangle + ... + \langle \phi_1,µ_0\rangle \},
\end{equation}
where the supremum is taken over all finite sequences satisfying $-µ_n \in A(\phi_n)$. The function $\psi$ is proper, convex and continuous over the separable Hilbert space $H^m(\mathbb{T}^d)$. Moreover, defining by $\partial \psi$ the sub-differential of $\psi$, $-A \subset \partial \psi$ by construction. Hence the result is proved since $\psi$ is Fr\'echet differentiable on a dense subset of the Hilbert space $H^m$ (since it is convex and continuous).
\end{proof}
\begin{Rem}
The result is stated for $m > 6d$ so that \color{black}$ H^m(\mathbb{T}^d)\subset \mathcal{C}^{2}(\mathbb{T}^d) $\color{black}. This point will be used directly later on in the paper as we shall directly consider $\mathcal{C}^2$ functions when needed.
\end{Rem}

\subsection{On the choice of writing the master equation in $\mptd$}
Before passing to the core Section of this paper, we take some time to comment the modeling choice we make to write the master equation on $\mptd$ instead of on $\{m \in \mmtd | m \geq 0\}=: \mathcal{M}_+(\mathbb{T}^d)$. Because, in the problem we are interested in, the number (or mass) of players stays constant, it is natural to consider the master equation only on $\mptd$, even if this situation is not the most general one. For instance, one can think about optimal stopping problem such as in \citep{bertucci2018optimal,bertucci2020monotone}. On the other hand, it is natural to define a value for the MFG whatever the total mass of players is. Of course in the situation of interest here, we can write the master equation on $\mathcal{M}_+(\mathbb{T}^d)$ and only the derivatives in the space of measures in directions which preserve the mass of the measure are needed. This previous remark makes the extension from $\mptd$ to $\mathcal{M}_+(\mathbb{T}^d)$ relatively easy. Moreover, working on the whole $\mathcal{M}_+(\mathbb{T}^d)$ \color{black} makes \color{black} easier to treat the question of uniqueness of solutions. For instance, recalling the proof of Proposition \ref{uniqsmooth} and its notations, the non-negativity of the function $W$ on $\mathcal{M}_+(\mathbb{T}^d)$ would have been sufficient to conclude. (This remark has higher implications later on in the paper.)

However, even though it seems more profitable to work on $\mathcal{M}_+(\mathbb{T}^d)$ than on $\mptd$, we prefer the second option as it allows us to use some existing results of the literature. We apologize for this inconvenience and hope that the interested reader shall be able to extend quite easily the results of this paper to the case of $\mathcal{M}_+(\mathbb{T}^d)$.

\section{Monotone solutions}\label{sec:mon}
In this section, we extend the notion of monotone solution introduced in \citep{bertucci2020monotone} to the equations \eqref{sme} and \eqref{me}. We shall not be concerned with the existence of such solutions here, as we delay this question to the next section. We start this section with the case of \eqref{sme} before treating \eqref{me}.

\subsection{The stationary case}
Even though we refer to \citep{bertucci2020monotone} for more details on why the notion of monotone solution is natural for MFG master equations, let us briefly recall the main idea behind this notion. 

The proof of Proposition \ref{uniqsmooth} clearly suggests that uniqueness of solutions can be obtained by looking at points of minimum of a function $W$ defined by $W = \langle U(µ)-V(\nu),µ - \nu\rangle$ for $U$ and $V$ two solutions. We then use the information that one has from the fact that $U$ and $V$ solve a master equation to proceed with the proof.

An important remark is that, at points of minimum of $W$, if $W$ is smooth, we have a relation to express some terms involving the derivatives of $U$ and $V$, uniquely through $U$ and $V$ (without derivatives). This is observed by combining \eqref{reluw} and Proposition \ref{firstcond}, or formally by taking $\frac{\delta W}{\delta µ} = 0$ in \eqref{reluw}. 

Hence, from the point of view of $U$, we only need information at points of minimum of $µ \to \langle U(µ) - V, µ - \nu\rangle$, for some function $V \in \mathcal{C}^2(\mathbb{T}^d)$ and measure $\nu \in \mptd$. But at these points of minimum, the terms involving the derivatives of $U$ with respect to $µ$ in \eqref{sme} can be expressed without using derivatives in the space of measures. This leads us to the following notion of solution of \eqref{sme}.
\begin{Def}\label{defstat}
A function $U \in \mathcal{B}$ is a monotone solution of \eqref{sme} if for any $\mathcal{C}^2$ function $\phi : \mathbb{T}^d \to \mathbb{R}$, for any measure $\nu \in \mmtd$, any point $m_0$ of strict minimum of $m \to \langle U(\cdot,m)- \phi, m- \nu \rangle$, the following holds
\begin{equation}
\begin{aligned}
r \langle U(\cdot,m_0), m_0 - \nu \rangle& +  \langle -\sigma \Delta U + H(\cdot,\nabla_x U), m_0 - \nu \rangle \geq \langle f(\cdot,m_0),m_0 - \nu\rangle\\
& -\langle U - \phi, \text{div}(D_pH(\nabla_x U)m_0)\rangle  - \sigma \langle \Delta(U- \phi),m_0\rangle.
\end{aligned}
\end{equation}
\end{Def}
\begin{Rem}
We only ask for information at points of strict minimum for stability reasons, mainly because a point of strict minimum can be approximated by points of strict minimum. Indeed in this context, thanks to Lemma \ref{Stegall}, given a point $m^*$ of strict minimum of a function $f : \mptd \to \mathbb{R}$, and a sequence $(f_n)_{n\geq 0}$ converging uniformly toward $f$, there always exists a sequence $(\tilde{f}_n)_{n \geq 0}$ such that for any $n \geq 0$, $\tilde{f}_n$ has a strict minimum $m_n$, $\tilde{f}_n$ is as closed as we want (in the uniform topology) of $f_n$ and $(m_n)_{n \geq 0}$ converges toward $m^*$.
\end{Rem}
This notion of monotone solution is reminiscent of the notion of viscosity solution introduced by Crandall and Lions in \citep{crandall1983viscosity}, although the equation \eqref{sme} does not have a proper comparison principle.

Let us also remark that Definition \ref{defstat} demands regularity in the space variable $x$. We shall not comment a lot on this except for the fact that, it seems possible to consider less regular functions in the space variable. In our particular framework, because of the presence of the i.i.d. noise, this regularity is not an issue.\\

The two following results justify in some sense the notion of solution we propose. The first one states that classical solutions are also monotone solutions, and the second one that there is uniqueness of a monotone solution in the monotone regime. 
\begin{Prop}
Assume that U is a classical solution of \eqref{sme}, then it is also a monotone solution of \eqref{sme}.
\end{Prop}
\begin{proof}
Consider a smooth solution $U$ of \eqref{sme}, a $\mathcal{C}^2$ function $\phi : \mathbb{T}^d \to \mathbb{R}$, $\nu \in \mathcal{M}(\mathbb{T}^d)$ and $m_0 \in \mptd$ point of strict minimum of $m \to \langle U(\cdot,m)- \phi, m- \nu \rangle =: W(m)$.

Up to a constant of $x$, the following holds for any $m \in \mptd, x \in \mathbb{T}^d$
\begin{equation}
\frac{\delta W}{\delta m}(m,x) = U(x,m) - \phi(x) + \left \langle \frac{\delta U}{\delta m}(\cdot,m,x), m - \nu\right \rangle.
\end{equation}

On the other hand, integrating \eqref{sme} against $m - \nu$ leads to : for all $m \in \mptd$,
\begin{equation}
\begin{aligned}
\langle &rU - \sigma \Delta U + H(x,\nabla_x U), m - \nu \rangle - \sigma \left \langle m - \nu \left | \frac{\delta U}{\delta m}(\cdot,m,\cdot)\right| \Delta m\right\rangle \\
& - \left \langle m- \nu \left | \frac{\delta U}{\delta m}(\cdot,m,\cdot) \right| \text{div}(D_pH(\cdot,\nabla_x U(\cdot,m))m)\right\rangle = \langle f(\cdot,m),m- \nu\rangle.
\end{aligned}
\end{equation}
From this we obtain
\begin{equation}
\begin{aligned}
\langle &rU - \sigma \Delta U + H(x,\nabla_x U), m - \nu \rangle - \left \langle  \frac{\delta W}{\delta m}(m,\cdot) ,\text{div}(D_pH(\cdot,\nabla_x U(\cdot,m))m)\right\rangle -\\
&- \sigma \left \langle  \frac{\delta W}{\delta m}(m,\cdot), \Delta m\right\rangle  = \langle f(\cdot,m),m- \nu\rangle - \sigma \langle U(\cdot,m) - \phi,\Delta m\rangle-\\
&\qquad\qquad\qquad\qquad \qquad \quad-\langle U(\cdot,m) - \phi,\text{div}(D_pH(\cdot,\nabla_xU(\cdot,m))m)\rangle.
\end{aligned}
\end{equation}
Evaluating this expression at $m_0$ and using Proposition \ref{firstcond} yields the required result.
\end{proof}

We can prove the following.
\begin{Theorem}\label{uniqthmstat}
Under Hypothesis \ref{hypmon}, two monotone solutions of \eqref{sme} in the sense of Definition \ref{defstat} only differ by a function $c : \mptd \to \mathbb{R}$. If a monotone solution exists it is a monotone application. 
\end{Theorem}
\begin{proof}
Let us assume that there exists two such solutions $U$ and $V$. Let us define $W :\mptd^2 \to \mathbb{R}$ with 
\begin{equation}
W(µ,\nu) = \langle U(\cdot,µ) - V(\cdot, \nu), µ-\nu\rangle := \int_{\mathbb{T}^d} U(x,µ) - V(x,\nu) (µ- \nu)(dx).
\end{equation}
We want to show that $W$ is a non-negative function. Assume that this is not the case and that there exists $(µ_1,\nu_1)$ such that $W(µ_1,\nu_1)< 0$. From this we deduce that there exists $\epsilon > 0$ such that for all $\phi,\psi \in \mathcal{C}^0(\mathbb{T}^d)$, $\|\phi\|_0 + \| \psi\|_0 \leq \epsilon$
\begin{equation}
\inf_{(µ,\nu) \in \mathcal{P}(\mathbb{T}^d)^2}\left\{ W(µ,\nu) + \langle \phi,µ\rangle + \langle \psi, \nu\rangle\right\} < \frac{W(µ_1,\nu_1)}{2}< 0.
\end{equation}
On the other hand, from Lemma \ref{Stegall}, we deduce that there exist $\phi,\psi \in \mathcal{C}^2(\mathbb{T}^d)$, $\|\phi\|_2 + \| \psi\|_2 \leq \epsilon$ such that $(µ,\nu) \to W(µ,\nu) + \langle \phi,µ\rangle + \langle \psi, \nu\rangle$ has a strict minimum at $(µ_0,\nu_0)$ on $\mathcal{P}(\mathbb{T}^d)^2$. Using the definition of monotone solutions for $U$, we deduce that
\begin{equation}
\begin{aligned}
r \langle U(\cdot,µ_0), µ_0 - \nu_0 \rangle& + \langle - \sigma \Delta U + H(\cdot,\nabla_x U), µ_0 - \nu_0 \rangle \geq \langle f(\cdot,µ_0),µ_0 - \nu_0\rangle\\
& -\langle U(µ_0) - V(\nu_0) + \phi, \text{div}(D_pH(\nabla_x U)µ_0)\rangle  - \sigma \langle \Delta(U- V + \phi),µ_0\rangle,
\end{aligned}
\end{equation}
and the corresponding relation for $V$ :
\begin{equation}
\begin{aligned}
r \langle V(\cdot,\nu_0), \nu_0 - \mu_0 \rangle& + \langle - \sigma \Delta V + H(\cdot,\nabla_x V), \nu_0 - \mu_0 \rangle \geq \langle f(\cdot,\nu_0),\nu_0 - \mu_0\rangle\\
& -\langle  V(\nu_0) - U(\mu_0) + \psi, \text{div}(D_pH(\nabla_x V)\nu_0)\rangle  - \sigma \langle \Delta(V- U + \psi),\nu_0\rangle.
\end{aligned}
\end{equation}
Combining the two previous relations, using the convexity of $H$ and the monotonicity of $f$ we deduce that
\begin{equation}
rW(µ_0,\nu_0) \geq- \langle  \phi, \text{div}(D_pH(\nabla_x U)µ_0)\rangle - \sigma \langle \Delta\psi,\nu_0\rangle - \langle \psi, \text{div}(D_pH(\nabla_x V)\nu_0)\rangle - \sigma \langle \Delta \phi,\mu_0\rangle,
\end{equation}
which is a contradiction because $\phi$ and $\psi$ can be chosen arbitrary small. Hence, we obtain that $W \geq 0$. This establishes the Theorem.\\

\end{proof}
\begin{Rem}
At this point, it can be disappointing that the uniqueness of monotone solutions can only be obtained up to this $c: \mptd \to \R$ function. In Section \ref{sec:extensionuniqueness}, we explain how this difficulty can be overcome by making additional assumptions or by changing slightly the setting of the master equation. Moreover, in some sense, this weak uniqueness result is sufficient to obtain (formally) the uniqueness of a Nash equilibria. Indeed the optimal strategy, given at the equilibrium by the solution $U$ of the master equation, only depends on the gradient in the spatial variable $x\in \mathbb{T}^d$ of $U$. Hence the addition of a function $c : \mptd \to \mathbb{R}$ does not alter the induced strategies.
\end{Rem}

We now give a result of stability of monotone solutions.
\begin{Prop}\label{stabs}
Assume that there exist sequences $(H_n)_{n \geq 0}$ and $(f_n)_{n \geq 0}$ in respectively $\mathcal{C}\color{black}^1\color{black}(\mathbb{T}^d \times \mathbb{R}^d, \mathbb{R})$ and $\mathcal{C}(\mathbb{T}^d \times \mptd, \mathbb{R})$ which converge locally uniformly (in those spaces) toward respectively $H$ and $f$. Assume that there is a sequence $(U_n)_{n \geq 0}$ of monotone solutions of \eqref{sme} (where $U_n$ is the solution associated with $H_n$ and $f_n$). Assume that $(U_n)_{\geq 0}$ converges locally uniformly toward some function $U \in \mathcal{B}^c$ (for the natural topology of $\mathcal{B}$), then $U$ is a monotone solution of \eqref{sme} associated with $H$ and $f$.
\end{Prop}
\begin{proof}
Let us consider $\phi \in \mathcal{C}^2$, $\nu \in \mmtd$ and $\mu_*$ a point of strict minimum of $m \to \langle U(\cdot,m) - \phi, m - \nu\rangle$ on $\mptd$. From Lemma \ref{Stegall}, we can consider a sequence of functions $(\phi_n)_{n \geq 0}$ such that $\|\phi_n\|_2 \to 0$ as $n \to \infty$ and $m \to \langle U_n(\cdot,m) - \phi + \phi_n, m - \nu\rangle$ admits a strict minimum at $µ_n$ on $\mptd$. Because $U_n$ is a monotone solution of \eqref{sme}, we obtain that 
\begin{equation}\label{interstabs1}
\begin{aligned}
r \langle U_n(\cdot,µ_n), µ_n - \nu \rangle& +  \langle -\sigma  \Delta U_n + H_n(\cdot,\nabla_x U_n), µ_n - \nu \rangle \geq \langle f_n(\cdot,µ_n),µ_n - \nu\rangle\\
& -\langle U_n - \phi + \phi_n, \text{div}(D_pH_n(\nabla_x U_n)µ_n)\rangle  - \sigma \langle \Delta(U_n- \phi + \phi_n),µ_n\rangle.
\end{aligned}
\end{equation}
Since $(µ_n)_{n \geq 0}$ is a compact sequence, extracting a subsequence if necessary, it converges toward a measure $\tilde{µ}$. By construction of $(µ_n)_{n \geq 0}$ and convergence of $(U_n)_{n \geq 0}$ toward $U$, we deduce that for any $m \in \mptd$,
\begin{equation}
\langle U(\cdot,\tilde{µ}) - \phi, \tilde{µ} - \nu\rangle \leq \langle U(\cdot,m) - \phi, m - \nu\rangle.
\end{equation}
From which we obtain $\tilde{µ} = µ_*$ (and the convergence of the whole sequence $(µ_n)_{n \geq 0}$). It now remains to pass to the limit $n \to \infty$ in \eqref{interstabs}. Remark that the terms in $\langle \Delta U_n(µ_n),µ_n\rangle$ cancel each other and that since $(U_n)_{n \geq 0}$ converges toward $U$ in $\mathcal{B}$, then $\|\Delta(U_n - U)\|_\infty \to 0$ as $n \to \infty$, thus, we can actually pass to the limit in all the terms in \eqref{interstabs1}, thanks to $U \in \mathcal{B}^c$, to obtain 
\begin{equation}
\begin{aligned}
r \langle U(\cdot,µ_*), µ_* - \nu \rangle& + \langle - \sigma\Delta U + H(\cdot,\nabla_x U), µ_* - \nu \rangle \geq \langle f(\cdot,µ_*),µ_* - \nu\rangle\\
&- \langle U_* - \phi , \text{div}(D_pH(\nabla_x U)µ_*)\rangle  - \sigma \langle \Delta(U- \phi),µ_*\rangle.
\end{aligned}
\end{equation}
Hence $U$ is a monotone solution of \eqref{sme} (associated with $H$ and $f$).
\end{proof}
\begin{Rem}
Remark that in this result, $U$ is asked to be slightly more regular than in the definition of a monotone solution.
\end{Rem}

We now give a result concerning smooth monotone solutions.
\begin{Prop}
Assume that $U$ is a $\mathcal{C}^1$ monotone solution such that $\dmU(x,m,y)$ is globally Lipschitz continuous, and its second order derivatives with respect to $y$ are continuous in $x,m$. Assume also that there exists $c >0$ such that
\begin{equation}\label{hyp552}
\langle U(m)- U(m'), m-m'\rangle \geq c (\textbf{d}_1(m,m'))^2.
\end{equation}
Then $U$ is a classical solution of \eqref{sme}.
\end{Prop}

\begin{proof}
Let us fix $\bar{m} \in \mptd$. If we choose $\nu$ sufficiently close to $\bar{m}$ and define $V$ by
\begin{equation}
V(x) = U(x,\bar{m}) + \left \langle \dmU(\cdot,\bar{m},x), \bar{m}- \nu\right \rangle,
\end{equation}
then $W : µ \to \langle U(µ) - V, µ - \nu\rangle$ has a strict minimum at $µ = \bar{m}$. 
Indeed, we can compute
\be
\begin{aligned}
W(µ) - W(\bar{m}) &= \langle U(µ) - V, µ - \nu\rangle + \left\langle \bar{m} - \nu\bigg| \dmU(\cdot,\bar{m},\cdot)\bigg| \bar{m}-\nu\right \rangle\\
& = \langle U(µ) - U(\bar{m}), µ - \bar{m} \rangle + \langle U(\mu) - U(\bar{m}), \bar{m}-\nu\rangle - \left\langle µ - \nu\bigg| \dmU(\cdot,\bar{m},\cdot)\bigg| \bar{m}-\nu\right \rangle\\
& \quad \quad + \left\langle \bar{m} - \nu\bigg| \dmU(\cdot,\bar{m},\cdot)\bigg| \bar{m}-\nu\right \rangle\\
&\geq c \textbf{d}_1^2(µ,\bar{m})+ \langle U(\mu) - U(\bar{m}), \bar{m}-\nu\rangle - \left\langle µ - \bar{m}\bigg| \dmU(\cdot,\bar{m},\cdot)\bigg| \bar{m}-\nu\right \rangle\\
& \quad \quad - \left\langle \bar{m} - \nu\bigg| \dmU(\cdot,\bar{m},\cdot)\bigg|\nu - \bar{m}\right \rangle+\left\langle \bar{m} - \nu\bigg| \dmU(\cdot,\bar{m},\cdot)\bigg|\bar{m} - \nu\right \rangle\\
\end{aligned}
\ee
The last two terms are non-negative since $U$ is monotone. Moreover, from the Lipschitz continuity assumption on $U$, we obtain finally
\be
W(µ) - W(\bar{m}) \geq  c \textbf{d}_1^2(µ,\bar{m}) - C\textbf{d}_1(\bar{m},\nu)\textbf{d}_1^2(µ,\bar{m}).
\ee
Hence the fact that $\bar{m}$ is a point of strict minimum of $W$ since $\nu$ can be chosen arbitrary close to $\bar{m}$. Since $U$ is a monotone solution, we then arrive at the relation 

\begin{equation}
\begin{aligned}
r \langle U(\cdot,\bar{m}), \bar{m} - \nu \rangle& +  \langle -\sigma \Delta U + H(\cdot,\nabla_x U), \bar{m} - \nu \rangle \geq \langle f(\cdot,\bar{m}),\bar{m} - \nu\rangle\\
& -\left\langle \bar{m}- \nu\bigg| \dmU(\cdot,\bar{m},\cdot)\bigg|\sigma \Delta\bar{m} + \text{div}(D_pH(\nabla_x U)\bar{m})\right\rangle
\end{aligned}
\end{equation}
which holds for any $\bar{m} \in \mptd, \nu \in \mptd, \textbf{d}_1(\nu,\bar{m}) < \epsilon$ for a given $\epsilon > 0$. In particular, for $\bar{m}$ with a smooth density, everywhere positive, we obtain easily that \eqref{sme} is satisfied everywhere. We finally deduce that \eqref{sme} is satisfied everywhere by a density argument, since, from the smoothness assumption we made on $U$, the equation is continuous in $m$.
\end{proof}

\begin{Rem}
Let us insist on the fact that, in practical situations, if we are given a smooth function which is also a monotone solution of \eqref{sme}, it is often the case that there is additional information on that function so that other proofs are available which avoids the use of \eqref{hyp552} which can seem somehow artificial.
\end{Rem}

\subsection{Toward the obtention of a more precise uniqueness result}\label{sec:extensionuniqueness}
Since monotone solutions were built on a proof of uniqueness, it is disappointing that Theorem \ref{uniqthmstat} does not give the full equality $U = V$. In our opinion, this is mainly due to our decision of writing the master equation on $\mptd$ instead of on $\mathcal{M}_1(\T^d):=\{m \in \mmtd, m \geq 0, \int m \leq 1\}$. We indicate in this Section how we can overcome this issue in two different settings.

\subsubsection{The case of a stronger monotonicity} Assume that, instead of  Hypothesis \ref{hypmon}, the coupling $f$ satisfies the stronger requirement 
\be\label{hypmon+}
\forall µ,\nu \in \mptd, \langle f(µ)(\cdot) - f(\nu)(\cdot),µ - \nu\rangle \leq 0 \Rightarrow µ = \nu.
\ee
Then, we can prove the following.
\begin{Theorem}
Assume that Hypothesis \ref{hypmon} holds as well as \eqref{hypmon+}. Then, there exists at most one monotone solution $U$ of \eqref{sme} such that $U$ is continuous, seen as a function $ \mptd \to \mathcal{C}^1$.
\end{Theorem}
\begin{proof}
Consider two such solutions $U$ and $V$. Following the proof of Theorem \ref{uniqthmstat}, we already know that $\nabla_x U = \nabla_x V$, and in particular that for all $µ,\nu \in \mptd$, $\langle U(µ)-V(\nu),µ- \nu\geq 0.$ Denote now by $\rho$ the Lebesgue measure on $\T^d$. Consider $W$ defined on $\mptd^2$ by
\be
W(µ,\nu) = \langle U(\cdot,µ) - V(\cdot,\nu),µ- \nu + \epsilon \rho\rangle,
\ee
for $\epsilon >0$. Assume that $U \ne V$, then there exists $\bar{m} \in \mptd$ such that $W(\bar{m},\bar{m}) \ne 0$. Hence, up to exchanging $U$ and $V$, there exists $\kappa > 0$ such that
\be\label{eq590}
\inf_{µ,\nu} W(µ,\nu) \leq -\kappa \epsilon.
\ee
Hence, using Stegall's Lemma, there $\phi,\psi \in \mathcal{C}^2$, arbitrary small, such that $µ, \nu \to W(µ,\nu) + \langle \phi,µ\rangle + \langle \psi,\nu\rangle$ has a strict minimum at $(µ_\epsilon,\nu_\epsilon)$. Since $U$ and $V$ are monotone solutions of \eqref{sme}, we obtain that
\be
\begin{aligned}
\langle r U(\cdot,µ_\eps)& - \sigma \Delta U(µ_\eps) + H(\cdot,\nabla_x U(µ_\eps)), µ_\eps - \nu_\eps + \eps \rho \rangle \geq \langle f(\cdot,µ_\eps),µ_\eps - \nu_\eps + \eps \rho \rangle\\
& -\langle U(µ_\eps) - V(\nu_\eps) + \phi , \text{div}(D_pH(\nabla_x U)µ_\eps)\rangle  - \sigma \langle \Delta(U- V + \phi),µ_\eps\rangle,
\end{aligned}
\ee
and
\be
\begin{aligned}
\langle r V(\cdot,\nu_\eps)& - \sigma \Delta V(\nu_\eps) + H(\cdot,\nabla_x V(\nu_\eps)), \nu_\eps - \mu_\eps - \eps \rho \rangle \geq \langle f(\cdot,\nu_\eps),\nu_\eps - \mu_\eps - \eps \rho \rangle\\
& -\langle V(\nu_\eps) - U(\mu_\eps) + \psi , \text{div}(D_pH(\nabla_x V)\nu_\eps)\rangle  - \sigma \langle \Delta(V- U + \psi),\nu_\eps\rangle.
\end{aligned}
\ee
Summing the two previous relations, and using the convexity of $H$ yields
\begin{equation}\label{intuniq2}
\begin{aligned}
rW&(µ_{\epsilon},\nu_{\epsilon}) + \epsilon\sigma \langle U(µ_\eps) - V(\nu_\eps), \Delta \rho\rangle + \langle H(\nabla_x U(µ_\eps)) - H(\nabla_x V(\nu_\eps)),\eps \rho\rangle\\
  &\geq\langle f(µ_{\epsilon}) - f(\nu_{\epsilon}),µ_{\epsilon}-\nu_{\epsilon} +\epsilon \rho\rangle- \langle  \phi_{\epsilon}, \text{div}(D_pH(\nabla_x U)µ_{\epsilon})\rangle - \sigma \langle \Delta\psi_{\epsilon},\nu_{\epsilon}\rangle&\\
 &\quad - \langle \psi_{\epsilon}, \text{div}(D_pH(\nabla_x V)\nu_{\epsilon})\rangle - \sigma \langle \Delta \phi_{\epsilon},\mu_{\epsilon}\rangle.
\end{aligned}
\end{equation}

Remark that the second term of the left side vanishes since $\Delta \rho = 0$. Consider now $(µ_0,\nu_0)$ such that, up to a subsequence, $(µ_\eps,\nu_\eps)_{\eps > 0}$ converges toward $(µ_0,\nu_0)$ (recall that $\mptd$ is compact). Passing to the limit $\epsilon \to 0$ in \eqref{intuniq2} implies that, since $f$ is continuous,
\be
\langle f(µ_0)-f(\nu_0),µ_0-\nu_0\rangle \leq r\langle U(µ_0)-V(\nu_0),µ_0- \nu_0\rangle \leq 0.
\ee
Hence, using \eqref{hypmon+}, we obtain that $µ_0 = \nu_0$. Hence, we can rewrite \eqref{intuniq2} in 
\be
r W(µ_\eps,\nu_\eps) \geq \epsilon o(\epsilon)
\ee
in the limit $\epsilon \to 0$. From this we contradict \eqref{eq590}. Hence the required result.
\end{proof}
In the previous proof, the additional requirement on $f$ can be too restrictive in several situations, whereas the additional assumptions on the regularity of $U$ is somehow standard, see for instance Proposition \ref{holderU}. In any case, we now give what we believe to be the natural way to address the question of the full uniqueness, if the previous raises difficulties in some situations.

\subsubsection{Writing the equation on $\mone$}
We detail what happens to the uniqueness proof when the master equation \eqref{sme} is written on $\mone$ instead of on $\mptd$. We do not reproduce the associated definition of monotone solution in this case, but we use exactly the same, except that we replace each time $\mptd$ by $\mone$. 

Reproducing the proof of Theorem \ref{uniqthmstat} in the same spirit by replacing $\mptd$ by $\mone$, we arrive at the conclusion that for all $µ, \nu \in \mone$, for any two monotone solutions $U$ and $V$, it holds that $\langle U(µ) - V(\nu),µ-\nu\rangle$. Hence, considering any $µ \in \mone$ such that $\int µ \in (0,1)$ we obtain that there exists $\eps > 0$ such that for any $\nu \in \mptd$,
\be
\epsilon \langle U(µ + \epsilon \nu) - V(µ), \nu \rangle \geq 0.
\ee
Dividing by $\eps$ and using the continuity of $U$, we obtain that for any $\nu \in \mptd, \langle U(µ) - V(µ),\nu\rangle \geq 0$. Thus we deduce that $U(µ) = V(µ)$ by symmetry. Hence we have proven the following.
\begin{Theorem}
Assume that $f$ is defined and monotone on $\mone$. Then, if $H$ is convex in $p$, there is at most one monotone solution of \eqref{sme} on $\mone$.
\end{Theorem}
The only restrictive point of this approach is that it is not obvious that we can extend $f$ to the whole $\mone$, even if in most of the practical applications, it seems to be the case. Moreover, we could have considered a smaller set than $\mone$ by looking at $\{m \in \mmtd, m \geq 0, \int m \in [a,b]\}$ for any $a < b$ such that $1 \in [a,b]$.

As we already mentioned above, we did not work in this setting for all the paper to avoid having to adapt the results of the existing literature, as they are almost all set on $\mptd$.

\subsection{The time dependent case}
Let us now introduce the definition of monotone solution in the time dependent setting. The approach is extremely similar except for the fact that, because we do not want to ask for time regularity outside of continuity for a solution $U$, we use techniques of viscosity solutions to treat the time derivative.
\begin{Def}\label{deft}
A function $U \in \mathcal{B}_t$ is a monotone solution of \eqref{me}  if \begin{itemize} \item for any $\mathcal{C}^2$ function $\phi : \mathbb{T}^d \to \mathbb{R}$, for any measure $\nu \in \mmtd$, for any smooth function $\vartheta : [0,\infty) \to \mathbb{R}$, $T > 0$ and any point $(t_0,m_0) \in (0,T]\times\mptd$ of strict minimum of $(t,m) \to \langle U(t,\cdot,m)- \phi, m- \nu \rangle - \vartheta(t)$ on $[0,T]\times \mptd$ the following holds
\begin{equation}
\begin{aligned}
\frac{d \vartheta}{dt}(t_0)& +  \langle -\sigma \Delta U + H(\cdot,\nabla_x U), m_0 - \nu \rangle \geq \langle f(\cdot,m_0),m_0 - \nu\rangle\\
&- \langle U - \phi, \text{div}(D_pH(\nabla_x U)m_0)\rangle  - \sigma \langle \Delta(U- \phi),m_0\rangle.
\end{aligned}
\end{equation}
\item the initial condition holds
\begin{equation}
U(0,\cdot,\cdot) = U_0(\cdot,\cdot).
\end{equation}
\end{itemize}
\end{Def}
As we did in the stationary case, we now present results of consistency and uniqueness of such solutions. The consistency result is a straightforward extension of the analogue on the stationary equation. Moreover, let us recall that we postpone the question of existence to the next section.
\begin{Prop}
Assume that U is a smooth solution of \eqref{me}, then it is also a monotone solution of \eqref{me}.
\end{Prop}
\begin{Theorem}\label{uniqtime}
Under Hypothesis \ref{hypmon}, two monotone solutions of \eqref{me} in the sense of Definition \ref{deft} only differ by a function $c : [0,\infty)\times\mptd \to \mathbb{R}$. If a monotone solution $U$ exists, $U(t)$ is a monotone application for all $t \geq 0$.
\end{Theorem}
\begin{Rem}
The same kind of developments as in Section \ref{sec:extensionuniqueness} can be done to obtain additional uniqueness results in this time dependent case. We do not present them as they are straightforward extensions.
\end{Rem}
\begin{proof}
Let us consider $U$ and $V$ two such solutions. We define $W$ by
\begin{equation}
W(t,s,µ,\nu) = \langle U(t,\cdot,µ) - V(s,\cdot, \nu), µ-\nu\rangle := \int_{\mathbb{T}^d} U(t,x,µ) - V(s,x,\nu) (µ- \nu)(dx).
\end{equation}
We want to prove that $W(t,t,µ,\nu) \geq 0$ for all $t \geq 0$, $µ,\nu \in \mptd$. Assume it is not the case, hence there exists $t_*,\delta, \bar{\epsilon} > 0$, such that for all $\epsilon \in(0,\bar{\epsilon}), \alpha > 0, \phi, \psi \in C^2$ such that $\|\phi\|_2 + \| \psi \|_2 \leq \epsilon$ and $\gamma_1,\gamma_2 \in (\frac{\bar{\epsilon}}{2}, \bar{\epsilon})$, 
\begin{equation}\label{infuniqt}
\inf_{t,s \in [0,t_*], µ,\nu \in \mptd} \left\{ W(t,s,µ,\nu) + \langle \phi, µ\rangle + \langle \psi, \nu \rangle + \frac{1}{2 \alpha}(t-s)^2 + \gamma_1 t + \gamma_2 s \right\} \leq - \delta.
\end{equation}
From Lemma \ref{Stegall}, we know that there exists (for any value of $\alpha$) $\phi, \psi, \gamma_1$ and $\gamma_2$ such that $(t,s,µ,\nu) \to W(t,s,µ,\nu) + \langle \phi, µ\rangle + \langle \psi, \nu \rangle + \frac{1}{2 \alpha}(t-s)^2 + \gamma_1 t + \gamma_2 s$ has a strict minimum on $[0,t_*]^2 \times \mptd^2$ at $(t_0,s_0,\mu_0,s_0)$.\\

We assume first that $t_0 > 0$ and $s_0 > 0$. Using the fact that $U$ is a monotone solution of \eqref{me} we obtain that
\begin{equation}
\begin{aligned}
-\gamma_1 - \frac{t_0 - s_0}{\alpha}& +  \langle - \sigma \Delta U(µ_0) + H(\cdot,\nabla_x U), µ_0 - \nu_0 \rangle \geq \langle f(\cdot,µ_0),µ_0 - \nu_0\rangle\\
& -\langle U(t_0,µ_0) - V(s_0,\nu_0) + \phi, \text{div}(D_pH(\nabla_x U)µ_0)\rangle\\
&  - \sigma \langle \Delta(U(t_0,µ_0)- V(s_0,\nu_0) + \phi),µ_0\rangle,
\end{aligned}
\end{equation}
and similarly for $V$ :
\begin{equation}
\begin{aligned}
-\gamma_2 - \frac{s_0 - t_0}{\alpha}& +  \langle -\sigma\Delta V(s_0,\nu_0) + H(\cdot,\nabla_x V), \nu_0 - \mu_0 \rangle \geq \langle f(\cdot,\nu_0),\nu_0 - \mu_0\rangle\\
& -\langle V(s_0,\nu_0) - U(t_0,\mu_0) + \phi, \text{div}(D_pH(\nabla_x V)\nu_0)\rangle\\
&  - \sigma \langle \Delta(V(s_0,\nu_0)- U(t_0,\mu_0) + \phi),\nu_0\rangle.
\end{aligned}
\end{equation}
Summing the two previous relations, using the monotonicity of $f$ and the convexity of $H$, we deduce that
\begin{equation}
-\gamma_1 - \gamma_2 \geq -\langle  \phi, \text{div}(D_pH(\nabla_x U)µ_0)\rangle - \sigma \langle \Delta \psi,\nu_0\rangle - \langle \psi, \text{div}(D_pH(\nabla_x V)\nu_0)\rangle - \sigma \langle \Delta \phi,\mu_0\rangle.
\end{equation}
The previous relation is a contradiction (provided that $\epsilon$ had been chosen sufficiently small compared to $\bar{\epsilon}$).\\

Let us now turn to the case $t_0 = 0$ (the case $s_0 = 0$ being treated in exactly the same fashion). By construction $s_0$ satisfies $|s_0-t_0|\leq C \sqrt{\alpha}$ for some $C > 0$ independent of $\epsilon$. Thus choosing $\alpha > 0$ sufficiently small, we easily manage to contradict \eqref{infuniqt}. Hence we have proven that $W(t,t,µ,\nu) \geq 0$ for $t\geq 0, µ,\nu \in \mptd$, which itself proves the claim.

\end{proof}

We now give a result of stability of monotone solutions.
\begin{Prop}\label{stabtime}
Assume that there exist sequences $(H_n)_{n \geq 0}$ and $(f_n)_{n \geq 0}$ in respectively $\mathcal{C}\color{black}^1\color{black}(\mathbb{T}^d \times \mathbb{R}^d, \mathbb{R})$ and $\mathcal{C}(\mathbb{T}^d \times \mptd, \mathbb{R})$ which converge locally uniformly (in those spaces) toward respectively $H$ and $f$. Assume that there is a sequence $(U_n)_{n \geq 0}$ of monotone solutions of \eqref{me} (where $U_n$ is the solution associated with $H_n$ and $f_n$). Assume that $(U_n)_{\geq 0}$ converges locally uniformly toward some function $U\in \mathcal{B}_t^c$ (for the natural topology of $\mathcal{B}_t$), then $U$ is a monotone solution of \eqref{me} associated with $H$ and $f$.
\end{Prop}
\begin{proof}
Let us consider $T > 0,\phi \in \mathcal{C}^2$, $\nu \in \mmtd$ and a smooth function $\vartheta : \mathbb{R} \to \mathbb{R}$. Consider also $(t_*,\mu_*)$ a point of strict minimum of $(t,m) \to \langle U(\cdot,m) - \phi, m - \nu\rangle - \vartheta(t)$ on $[0,T]\times \mptd$. From Lemma \ref{Stegall}, we can consider a sequence of functions $(\phi_n)_{n \geq 0}$ and of real numbers $(\delta_n)_{n \geq 0}$ such that $\|\phi_n\|_2 +\delta_n \to 0$ as $n \to \infty$ and $(t,m) \to \langle U_n(\cdot,m) - \phi + \phi_n, m - \nu\rangle - \vartheta(t)- \delta_n t$ admits a strict minimum at $(t_n,µ_n)$ on $[0,T]\times\mptd$. Because $U_n$ is a monotone solution of \eqref{me}, we obtain that 
\begin{equation}\label{interstabs}
\begin{aligned}
\frac{d \vartheta}{dt}(t_n) + \delta_n& + \langle -\sigma \Delta U_n + H_n(\cdot,\nabla_x U_n), µ_n - \nu \rangle \geq \langle f_n(\cdot,µ_n),µ_n - \nu\rangle\\
& -\langle U_n - \phi + \phi_n, \text{div}(D_pH_n(\nabla_x U_n)µ_n)\rangle  - \sigma \langle \Delta(U_n- \phi + \phi_n),µ_n\rangle.
\end{aligned}
\end{equation}
Following the same arguments as in the proof of Proposition \ref{stabs} we obtain first that $(t_n,µ_n) \to (t_*,µ_*)$ and then that
\begin{equation}
\begin{aligned}
\frac{d \vartheta}{dt}(t_*)& + \langle -\sigma\Delta U + H(\cdot,\nabla_x U), µ_* - \nu \rangle \geq \langle f(\cdot,µ_*),µ_* - \nu\rangle\\
&- \langle U_* - \phi , \text{div}(D_pH(\nabla_x U)µ_*)\rangle  - \sigma \langle \Delta(U- \phi),µ_*\rangle.
\end{aligned}
\end{equation}
Hence $U$ is a monotone solution of \eqref{me}.
\end{proof}

\section{Existence of monotone solutions}\label{sec:exist}
In this section, we establish the existence of a monotone solution of \eqref{me}, in cases for which the assumptions of Theorem \ref{existclassic} are not satisfied. We first prove an estimate for classical solutions of \eqref{me} and then use a stability result to prove our existence result. In our opinion, this section plays an important (pedagogical) role by explaining why the stability of monotone solutions is crucial to obtain existence results. In particular we do so on an example which we believe to be natural and instructive. Our aim here is not to weaken optimally the regularity assumptions needed. Even if the following results are clear improvements of the existing literature, it seems that even more can be done in this direction. We can prove the following.
\begin{Prop}\label{holderU}
Assume that $U$ is a classical solution of \eqref{me} and that there exists $C > 0, \alpha,\beta \in (0,1)$ such that
\begin{itemize}
\item \begin{equation}
\sup_{ µ, \nu \in \mptd} \frac{\|f(µ)- f(\nu)\|_{\color{black}1 + \color{black}\alpha}}{\textbf{d}_1(µ,\nu)^{\beta}} + \frac{\|U_0(µ)- U_0(\nu)\|_{\color{black}2 \color{black}+ \alpha}}{\textbf{d}_1(µ,\nu)^{\beta}} \leq C.
\end{equation}
\item $H$ satisfies \eqref{hypH} with the same C.
\end{itemize}
Then there exists $C' >0$ depending only on $C, \alpha$ and $\beta$ such that
\begin{equation}\label{esthold}
|U(t,x,m) - U(t',x',m')| \leq C' (|t-t'|^{\frac{\gamma}{2}} + |x - x'| + \textbf{d}_1(m,m')^{\gamma}),
\end{equation}
where $\gamma = (2(\beta^{-1} - \frac{1}{2}))^{-1} \in (0,1)$.
\color{black}Moreover, 
\begin{equation}\label{eq781}
\|U(t,\cdot,m) - U(s,\cdot,m')\|_{2 + \alpha} \leq C' (|t-t'|^{\frac{\gamma}{2}} + \textbf{d}_1(m,m')^{\gamma})
\end{equation}
\end{Prop}
The following proof is similar to the one of a result in \citep{cardaliaguet2019master} which establishes the global Lipschitz regularity of $U$.
\begin{proof}
Lets us take $t\geq 0$, $µ_1,µ_2 \in \mptd$. We define for $i \in \{1;2\}$, $(u_i,m_i) \in \mathcal{C}^{1,2,\alpha}\times \mathcal{C}([0,T], \mptd)$ the unique solution (\citep{lasry2007mean}) of 
\begin{equation}
\begin{cases}
-\partial_t u_i - \sigma \Delta u_i + H(x,\nabla u_i) = f(x,m_i) \text{ in } [0,t]\times \mathbb{T}^d,\\
\partial_t m_i - \sigma \Delta m_i - \text{div}(D_p H(x,\nabla u_i)m_i) = 0 \text{ in } [0,t]\times \mathbb{T}^d,\\
m_i(0) = µ_i, u_i(t) = U_0(m_i(t)), \text{ in } \mathbb{T}^d.
\end{cases}
\end{equation}
Since $U$ is a classical solution of \eqref{me}, it follows that $U$ satisfies 
\begin{equation}
\forall x \in \mathbb{T}^d, U(t,x,µ_1) = u_1(0,x).
\end{equation}
From this and the regularity assumptions on $f$, $H$ and $U_0$, we deduce that there exists $C > 0$, such that for all $s \geq 0, µ \in \mptd$, $\|U(s,\cdot,µ)\|_{2+ \alpha} \leq C$, from which we deduce the estimate in the space variable in \eqref{esthold}.\\

We now come back to the H\"older estimate in the measure argument. Let us compute 
\begin{equation}
\begin{aligned}
0 \leq& \int_0^t \int_{\mathbb{T}^d} \left[ - \partial_t(u_1 - u_2) - \sigma \Delta (u_1 - u_2) + H(x,\nabla u_1) - H(x,\nabla u_2)\right] d(m_1(s) - m_2(s))ds,\\
=&  \int_0^t \int_{\mathbb{T}^d} (H(x,\nabla u_1) - H(x,\nabla u_2) + D_p H(x,\nabla u_1)\cdot \nabla (u_1 - u_2))dm_1ds\\
& + \int_0^t \int_{\mathbb{T}^d} (H(x,\nabla u_2) - H(x,\nabla u_1) + D_p H(x,\nabla u_2)\cdot \nabla (u_2 - u_1))dm_2ds\\
& - \int_{\mathbb{T}^d}(U_0(m_1(t)) - U_0(m_2(t)))d(m_1(t) - m_2(t)) + \int_{\mathbb{T}^d}(u_1(0) - u_2(0))d(µ_1 - µ_2).
\end{aligned}
\end{equation}
Here, we have used the monotonicity of $f$ for the inequality. Using the convexity of $H$ and the monotonicity of $U_0$, we obtain
\begin{equation}\label{stepone}
\int_0^t \int_{\mathbb{T}^d} |\nabla (u_1 - u_2)|^2d(m_1(s) + m_2(s))ds \leq C\int_{\mathbb{T}^d}(u_1(0) - u_2(0))d(µ_1 - µ_2) \leq C d_1(µ_1, µ_2).
\end{equation}
Let us remark that since $\nabla_x u_i$ is indeed uniformly bounded, we can use a strict-like convexity of $H$ to obtain the previous inequality. The estimate \eqref{stepone} is extremely helpful to establish the next estimate on the trajectories $m_1$ and $m_2$ that we now provide using a coupling argument. Let $X_1$ and $X_2$ be two random variables of law $µ_1$ and $µ_2$ such that $\mathbb{E}[|X_1 - X_2|] = d_1(µ_1,µ_2)$. Let us define $(X_{i,s})_{s \geq 0}$ for $i \in \{1;2\}$ the strong solutions of 
\begin{equation}
\begin{cases}
dX_{i,s} = -D_p H(X_{i,s},\nabla u_i(X_{i,s}))ds + \sqrt{2 \sigma} dB_s,\\
X_{i,0} = X_i,
\end{cases}
\end{equation}
for $(B_s)_{s \geq 0}$ a standard Brownian motion. We now compute using It\^{o}'s Lemma
\begin{equation}
\begin{aligned}
\mathbb{E}[|X_{1,s} - X_{2,s}|] \leq \mathbb{E}[|X_1 - X_2|] + \mathbb{E}\left[ \int_0^s |D_p H(X_{1,r}, \nabla u_1(X_{1,r})) - D_p H(X_{2,r}, \nabla u_1(X_{2,r}))|dr \right]\\
+\mathbb{E}\left[ \int_0^s |D_p H(X_{2,r}, \nabla u_1(X_{2,r})) - D_p H(X_{2,r}, \nabla u_2(X_{2,r}))|dr \right].
\end{aligned}
\end{equation}
We now deduce, using the Lipschitz continuity of $D_p H$ and the Lipschitz continuity of $\nabla u_1$ for the second term, and \eqref{stepone} for the third term, that
\begin{equation}
\mathbb{E}[|X_{1,s}-X_{2,s}|] \leq \mathbb{E}[|X_1 - X_2|] + C\int_0^s\mathbb{E}[|X_{1,r} - X_{2,r}|]dr + C \left( \int_{\mathbb{T}^d}(u_1(0) - u_2(0))d(µ_1 - µ_2)\right)^{\frac{1}{2}}.
\end{equation}
From which we deduce using Gronwall's Lemma that
\begin{equation}\label{steptwo}
\sup_{s \in [t,T]}d_1(m_1(s),m_2(s)) \leq C\left(d_1(µ_1,µ_2) +  \left( \int_{\mathbb{T}^d}(u_1(0) - u_2(0))d(µ_1 - µ_2)\right)^{\frac{1}{2}}\right).
\end{equation}
Let us now remark that, using Lemma 3.2.2 in \citep{cardaliaguet2019master}, we deduce
\begin{equation}
\sup_{s \in [0,t]}\|(u_1-u_2)(s)\|_{\color{black}2\color{black} + \alpha} \leq C\left(\sup_{s \in [0,t]}\|f(m_1(s))- f(m_2(s))\|_{\color{black} 1 + \color{black}\alpha} + \|U_0(m_1(t))- U_0(m_2(t))\|_{\color{black} 2 \color{black}+\alpha}\right).
\end{equation}
Hence we deduce, using \eqref{steptwo} and the assumptions on $f$ and $U_0$ that, that there exists $C > 0$ such that
\begin{equation}
\sup_{s \in [0,t]}\|(u_1-u_2)(s)\|_{\color{black}2\color{black} + \alpha} \leq C\left(\textbf{d}_1(µ_1,µ_2) + \|(u_1-u_2)(0)\|_{\color{black}2\color{black} + \alpha}^{\frac{1}{2}}\textbf{d}_1(µ_1,µ_2)^{\frac{1}{2}}\right)^{\beta}.
\end{equation}
We now easily obtain that \color{black}
\begin{equation}\label{halfd1}
\|U(t,\cdot,µ_1) - U(t,\cdot,µ_2)\|_{2 + \alpha} = \|u_1(0,\cdot) - u_2(0,\cdot)\|_{2 + \alpha} \leq C d_1(µ_1,µ_2)^{\gamma},
\end{equation}
\color{black}for $\gamma= (2(\beta^{-1} - \frac{1}{2}))^{-1} \in (0,1)$.
Let us now recall that in view of the Lipschitz continuity of $D_p H$, we have the classical estimate for the solution of the Fokker-Planck equation :
\begin{equation}
\forall s,s' \in [t,T], d_1(m_1(s),m_1(s')) \leq C \sqrt{|s- s'|},
\end{equation}
where $C$ is a constant independent of $µ_1$ and $t \geq 0$. Moreover, the following relation holds
\begin{equation}
\forall s \in [t,T], U(s,x,m_1(s)) = u_1(s,x).
\end{equation}
Recalling \eqref{steptwo}, we finally obtain that there exists $C > 0$ such that for any $t,s \in [0,T], x \in \mathbb{T}^d, µ \in \mptd$ \color{black}
\begin{equation}
\|U(t,\cdot,m) - U(s,\cdot,m)\|_{2 + \alpha} \leq C |t -s |^{\frac{\gamma}{2}},
\end{equation}
\color{black}which concludes the proof.
\end{proof}
\begin{Rem}
The extension of this result to value function\color{black}s \color{black}  being defined on $\mathcal{M}_+(\mathbb{T}^d) := \{m \in \mmtd| m \geq 0\}$ is straightforward when equipping the previous convex set with the metric $\tilde{\textbf{d}}_1(µ,\nu) := \sup \langle \phi, µ - \nu\rangle$ where the supremum is taken over Lipschitz functions on $\mathbb{T}^d$ whose Lipschitz constant is at most $1$ and which verify $\phi(0) = 0$.
\end{Rem}
\begin{Rem}\label{extensionlip}
Let us remark that, following exactly the proof of Proposition 3.2 in \citep{cardaliaguet2019master},  assuming 
\begin{equation}
Lip_{\alpha} (f) + Lip_{1+\alpha}(U_0) \leq C,
\end{equation}
instead of the H\"older continuity estimates in the previous proposition, one arrives at the conclusion that for some $C' > 0$ and any $(t,x,m),(t',x',m') \in [0,\infty)\times \mathbb{T}^d \times \mptd$
\begin{equation}
|U(t,x,m) - U(t',x',m')| \leq C' (|t-t'|^{\frac{1}{2}} + |x - x'| + \textbf{d}_1(m,m')).
\end{equation}
\end{Rem}
Having established this a priori estimate, we are now in position to prove the existence of monotone solutions which are not necessary classical solutions.
\begin{Theorem}\label{existmon}
Assume that Hypothesis \ref{hypmon} holds and that
\begin{itemize}
\item The Hamiltonian $H$ satisfies \eqref{hypH}.
\item $f$ and $U_0$ satisfy the assumptions of Proposition \ref{holderU} for some $\beta > 0$ and $f$ and $U_0$ are in the closure (with respect to the uniform convergence) of the set of couplings $f$ and $U_0$ satisfying the assumptions of Theorem \ref{existclassic}.
\end{itemize} Then, there exists a (unique) monotone solution of the master equation \eqref{me} in the sense of Definition \ref{deft}. 
\end{Theorem}
\begin{proof}
Let us consider sequences $(f_n)_{n \geq 0}$ and $(U_{0,n})_{n \geq 0}$ which approximate $f$ and $U_0$. For any $n \geq 0$, thanks to Theorem \ref{existclassic}, there exists a (unique) solution $U_n$ of \eqref{me} associated to $f_n$ and $U_{0,n}$. Using Proposition \ref{holderU}, we obtain that the sequence $(U_n)_{n \geq 0}$ is a uniformly continuous sequence of functions. Since it is bounded uniformly for $t = 0$, we deduce from Ascoli-Arzela Theorem that $(U_n)_{n \geq 0}$ is a compact sequence for the uniform convergence. Hence, extracting a subsequence if necessary, it converges toward a limit $U_*$. Since for all $n \geq 0$, $U_n$ is a classical solution of \eqref{me} (associated to $f_n$ and $U_{0,n}$) we deduce that it is also a monotone solution of the same equation. We finally conclude using Proposition \ref{stabtime} that $U_*$ is a monotone solution of \eqref{me}. Note that Proposition \ref{stabtime} can indeed be used because \eqref{eq781} holds uniformly in $n$.
\end{proof}
\color{black}
We do not want to enter into the complete problem of the regularization of functions from $\mptd$ to $C(\mathbb{T}^d)$, that is why we add the assumption that $f$ and $U_0$ have to be in the closure of more regular functions. This assumption seems somehow to be necessary. Indeed, even if there are ways to regularize in a general manner such functions, for instance the mollifiers introduced in \citep{mou}, the conservation of monotonicity through this process is not true in general.

Moreover, the set of couplings we are interested in is not trivial. Indeed, consider for instance a coupling $f$ defined on $\mathbb{T}^d\times \mptd$ by
\begin{equation}
f(x,m):= \int_{\mathbb{T}^d}\phi(z,(m\star\rho)(z))\rho(x-z)dz,
\end{equation}
where $\phi : \mathbb{R}^2 \to \mathbb{R}$ is a continuous function, smooth in its first argument and H\"older continuous in its second one, with $\rho$ a smooth non-negative even function. By regularizing $\phi$, one obtain a regularization of $f$. Hence such a coupling $f$ is the closure of couplings satisfying the assumptions of Theorem \ref{existclassic}. Moreover it satisfies the required assumptions.

We now give a sufficient condition which implies such an approximation property for the coupling. Note that it is independent from the previous example.
\begin{Prop}
Consider a monotone $f : \mptd \to \mathcal{C}(\mathbb{T}^d)$ and denote by $\rho$ a non-negative, even, $\mathcal{C}^{\infty}$ function on $\mathbb{T}^d$. For $\epsilon > 0$, define $\rho_{\epsilon}(\cdot) = \epsilon^{-d}\rho(\epsilon^{-1}\cdot)$. Assume that for some $k \geq 0$, the restriction of $f : \mptd \cap \mathcal{C}^k(\mathbb{T}^d) \to \mathcal{C}^{-k}(\mathbb{T}^d)$ is $C^{1,1}$, for the norms\footnote{$C^{-k}$ is the topological dual set of $C^{k}$ while we understand $C^{1,1}$ in the sense that the usual differential is a Lipschitz function.} of $C^k$ and $C^{-k}$. Define 
\begin{equation}
f^{\epsilon}(m)(x) = (\rho_{\epsilon}\star f(m\star\rho_{\epsilon}))(x).
\end{equation}
Then, $f^{\epsilon}$ satisfies the requirements of Theorem \ref{existclassic}, is monotone and if $f$ is continuous, $f^{\epsilon}\to f$ uniformly as $\epsilon \to 0$.
\end{Prop}
\begin{proof}
Consider $µ,\nu$ and compute
\begin{equation}
\begin{split}
\langle f^{\epsilon}(µ)-f^{\epsilon}(\nu),µ - \nu\rangle &= \langle f(µ\star \rho_{\epsilon}) - f(\nu \star\rho_{\epsilon}), µ\star \rho_{\epsilon} - \nu\star \rho_{\epsilon}\rangle\\
& \geq 0,
\end{split}
\end{equation}
since $f$ is monotone. The fact that $f^{\epsilon}$ has the required regularity follows from standard properties of the convolution product.
\end{proof}

\color{black}

\section{First order master equations with common noise}\label{sec:firstcom}
The interest of this section is twosome. First we want to present some structures of common noise in MFG, which, despite not being entirely new, have attracted little attention in the literature even though they cover a wide range of applications. Secondly, we explain, without entering into the same amount of details as we did for \eqref{sme} or \eqref{me}, how the notion of monotone solution is helpful for such cases. We mainly look at two situations. The first one in which an additional parameter is added which is stochastic and which affects all the players in the same way (like a price on a market for instance). The second one is a situation in which players jump in a coordinated manner at random times, as the one introduced in \citep{bertucci2019some}. The types of noises at interest in this section seem to be more used in the Economics literature \citep{scheinkman1986borrowing,krusell1998income,gabaix2016dynamics,achdou2022income} than studied in the mathematics one.

For pedagogical reasons, we introduce first a case in which the additional parameter has only two states.

As most of the following analysis does not rely on the fact that we are in a time dependent or a stationary situation, we choose here to focus on the time dependent setting.

Finally, let us state that this section is not particularly concerned with the existence of monotone or classical solutions of the master equations we shall introduce. We present in subsection \ref{subs:apriori} an a priori estimate which is valid in all the cases we shall introduce and only detail a proof of existence of monotone solutions in the case of common jumps.

\subsection{Additional two-state stochastic parameter}
Consider a MFG which is similar to the one presented in section \ref{sec:model} except for the fact that there is an additional parameter $p$, totally exogenous from the rest of the game, which can take two values $p_1$ and $p_2$. We assume that this parameter affects the player in the following way : when $p = p_i$ the running cost of the player is a function $f_i : \mathbb{T}^d \times \mptd \to \mathbb{R}$. We assume that $p$ is a random process which jump from $p_1$ to $p_2$ with a transition rate $\lambda_1 > 0$, and from $p_2$ to $p_1$ with a transition rate $\lambda_2$. In such a situation, the associated master equation is in fact the system of two master equations
\begin{equation}\label{me2p}
\begin{aligned}
\partial_t U_i& - \sigma \Delta U_i + H(x,\nabla_x U_i) - \left\langle \frac{\delta U_i}{\delta m}(x,m,\cdot), \text{div}\left(D_pH(\cdot, \nabla U_i(\cdot,m))m\right)\right\rangle\\
& - \sigma\left\langle  \frac{\delta U_i}{\delta m}(x,m,\cdot), \Delta m \right\rangle  + \lambda_i(U_i - U_j)= f_i(x,m) \text{ in } (0,\infty)\times \mathbb{T}^d\times \mptd,\\
U_i(0,&x,m) = U_0(x,m) \mathbb{T}^d\times \mptd,
\end{aligned}
\end{equation}
where $i,j \{1;2\}, i \ne j$. We do not derive precisely the previous equation as it is quite standard (at least given the derivation of the usual master equation). However, let us briefly comment on it. Consider the problem of a player in the MFG in the state $p= p_1$ facing almost the situation described above, except for the fact that, at the random times which described the switches of $p$, the game stops for the player and it is getting the value $\varphi(t,x,m)$ if this event occurs at time $t$, that it is in the position $x$ and the players it is facing are described by $m$. In such a situation, it is quite standard that the correct term to model such phenomenon should be $\lambda(U_1 - \varphi)$ instead of $\lambda (U_1 - U_2)$. To obtain \eqref{me2p}, it suffices to remark that here the game does not stop and $U_2$ plays the role of $\varphi$.
\begin{Rem}
Not only $f$, but $H,$ $\sigma$ and $U_0$ could have depended on $i$, without changing any of the following. We only restricted ourselves to the case of $f$ to lighten the notation.
\end{Rem}
Following section \ref{sec:mon}, we naturally propose the following definition of monotone solution for \eqref{me2p}.
\begin{Def}\label{defp2}
A pair of functions $(U_1,U_2) \in \mathcal{B}_t$ is a monotone solution of \eqref{me2p}  if \begin{itemize} \item for any $\mathcal{C}^2$ functions $(\phi_1, \phi_2) : \mathbb{T}^d \to \mathbb{R}$, for any measure $\nu \in \mmtd$, for any smooth function $\vartheta : [0,\infty) \to \mathbb{R}$, $T > 0$ and any point $(i_0,t_0,m_0) \in \{1;2\} \times(0,T]\times\mptd$ of strict minimum of $(i,t,m) \to \langle U_i(t,\cdot,m)- \phi_i, m- \nu \rangle - \vartheta(t)$ on $\{1;2\}\times [0,T]\times \mptd$, the following holds
\begin{equation}
\begin{aligned}
\frac{d \vartheta}{dt}(t_0)& +  \langle -\sigma \Delta U_{i_0} + H(\cdot,\nabla_x U_{i_0}), m_0 - \nu \rangle + \lambda_{i_0}\langle U_{i_0} - U_{j_0}, m_0 - \nu\rangle\\
 \geq& \langle f_{i_0}(\cdot,m_0),m_0 - \nu\rangle  - \langle U_{i_0} - \phi_{i_0}, \text{div}(D_pH(\nabla_x U_{i_0})m_0)\rangle\\
  &- \sigma \langle \Delta(U_{i_0}- \phi_{i_0}),m_0\rangle,
\end{aligned}
\end{equation} where $j_0 \ne i_0$.
\item the initial condition holds
\begin{equation}
U_1(0,\cdot,\cdot) = U_2(0,\cdot,\cdot)= U_0(\cdot,\cdot).
\end{equation}
\end{itemize}
\end{Def}
We now present a result of uniqueness of monotone solutions of \eqref{me2p}, which does not rely on any particular assumption on the evolution of the process $p$. This is not surprising since the uniqueness of solutions arises from monotonicity properties which models a sort of competition between the players, and this additional parameter does not perturb the competition between the players.
\begin{Theorem}
Under Hypothesis \ref{hypmon}, two pairs of functions $(U_1,U_2)$ and $(V_1,V_2)$, monotone solutions of \eqref{me2p} in the sense of Definition \ref{defp2}, only differ by a pair of functions $c_1,c_2 : [0,\infty) \times \mptd \to \mathbb{R}$.\end{Theorem}
\begin{Rem}
We here understand Hypotheses \ref{hypmon} in the sense that $f(\cdot,p_1)$ and $f(\cdot,p_2)$ both satisfy the monotonicity assumptions. As in the previous case, we could have adapted Section \ref{sec:extensionuniqueness} in this situation.
\end{Rem}
\begin{proof}
The proof of this result is very similar to the one of Theorem \ref{uniqtime}. Let us take two solutions $(U_1,U_2)$ and $(V_1,V_2)$ and define $W : \{1;2\}\times [0,T]^2 \times \mptd^2$ by
\begin{equation}
W_i(t,s,µ,\nu) = \langle U_i(t,µ) - V_i(s,\nu), µ - \nu \rangle.
\end{equation}
Let us now remark, following the same argument as in the proof of Theorem \ref{uniqtime}, that at a point $(i_0,t_0,s_0,µ_0,\nu_0)$ of strict minimum of $W$ (up to the addition small perturbations using Lemma \ref{Stegall}), the term arising in the relation of monotone solutions from the additional term in $\lambda_{i_0}$ is of the form
\begin{equation}
\lambda_{i_0}(W_{i_0}(t_0,s_0,µ_0,\nu_0) - W_{j_0}(t_0,s_0,µ_0,\nu_0))
\end{equation}
for $j_0 \ne i_0$. Because we are at a point of minimum of $W$, this term has a sign and the rest of the proof follows quite easily.
\end{proof}
\begin{Rem}
Even though they are true, we do not write once again precise results of stability or consistency for this particular master equation because they are merely trivial adaptations of the ones we already gave.
\end{Rem}

\subsection{Additional stochastic parameter following a stochastic differential equation}
We now place ourselves in the same framework as in the previous section, except for the fact that now the parameter $p$ is supposed to evolve according to 
\begin{equation}\label{sdep}
dp_t = b(p_t) dt + \sqrt{2\sigma'}dB_t
\end{equation}
where $(B_t)_{t \geq 0}$ is a standard $k$ dimensional Brownian motion, $b : \mathbb{R}^k \to \mathbb{R}^k$ a smooth function and $\sigma' \geq 0$. For simplicity, we assume that $p$ is valued in $\mathbb{T}^k$ (using the classical quotient $\mathbb{T}^k = \mathbb{R}^k / \mathbb{Z}^k$). This simplification does not play a role if for the fact that it simplifies the notation and the formulation of some results. Following the previous subsection on the two states case, we want to define the value function of the MFG as a function of $p$ (in addition to the other variables). Assuming now that $f$ is also a function of $p$, we naturally arrive at the master equation
\begin{equation}\label{mecp}
\begin{aligned}
\partial_t U& - \sigma \Delta_x U + H(x,\nabla_x U) - \sigma'\Delta_p U - b\cdot \nabla_p U - \sigma\left\langle \dmU(x,m,p,\cdot), \Delta m \right\rangle \\
-& \left\langle \dmU(x,m,p,\cdot), \text{div}\left(D_pH(\cdot, \nabla_x U(\cdot,m,p))m\right)\right\rangle  = f(x,m,p) \text{ in } (0,\infty)\times \mathbb{T}^d\times \mptd\times \mathbb{T}^k,\\
U(0&,x,m,p) = U_0(x,m,p) \text{ in } \mathbb{T}^d\times \mptd\times \mathbb{T}^k.
\end{aligned}
\end{equation}
Let us observe that the additional terms in this equation simply follows from the usual computation of the infinitesimal change of the value along the equilibrium path. For instance if the game is trivial and $U$ does not depend on $m$, then those terms are simply associated to the generator of the stochastic differential equation \eqref{sdep}.

In the following, we are not going to enter into much details about the regularity of $U$ with respect to $p$. Let us only remark that if $f$ and $U_0$ are smooth functions of $p$, satisfying, uniformly in $p$, the assumptions of Theorem \ref{existclassic}, then we expect that there exists a classical solution of the master equation \eqref{mecp}. The following definition should by now seems natural to the reader.
\begin{Def}\label{defcp}
A function $U$ in $\mathcal{B}'_t$ is a monotone solution of \eqref{mecp} if \begin{itemize} \item for any $\mathcal{C}^2$ function $\phi : \mathbb{T}^d\times \mathbb{T}^k \to \mathbb{R}^d$, for any measure $\nu \in \mmtd$, for any smooth function $\vartheta : [0,\infty) \to \mathbb{R}$, any $T >0$ and any point $(t_0,m_0,p_0) \in (0,T]\times\mptd\times \mathbb{T}^k$ of strict minimum of \begin{equation}
(t,m,p) \to \langle U(t,\cdot,m,p)- \phi(\cdot,p), m- \nu \rangle - \vartheta(t),\end{equation} on $[0,T]\times \mptd\times \mathbb{T}^k$, the following holds
\begin{equation}
\begin{aligned}
\frac{d \vartheta}{dt}(t_0)& +  \langle -\sigma\Delta_x U  - b\cdot \nabla_p U - \sigma'\Delta_p U+ H(\cdot,\nabla_x U), m_0 - \nu \rangle\\
 \geq& \langle f(\cdot,m_0,p_0),m_0 - \nu\rangle  - \langle U - \phi, \text{div}(D_pH(\nabla_x U)m_0)\rangle - \sigma \langle \Delta_x(U- \phi),m_0\rangle.
\end{aligned}
\end{equation}
\item The initial condition holds
\begin{equation}
U(0,\cdot,\cdot,\cdot) = U_0(\cdot,\cdot,\cdot).
\end{equation}
\end{itemize}
\end{Def}
As it was the case in the two states model, a result of uniqueness can be established without much assumptions on the evolution of the stochastic process $(p_t)_{t\geq 0}$ or on the dependence of $f$ on it.
\begin{Theorem}
Under Hypothesis \ref{hypmon}, two monotone solutions of  \eqref{mecp} in the sense of Definition \ref{defcp} differ only by a function $c : [0,\infty)\times \mptd \times \mathbb{T}^k \to \mathbb{R}$. If such a solution $U$ exists, for any $t \geq 0, p \in \mathbb{T}^k$, $U(t,\cdot,\cdot,p)$ is monotone.
\end{Theorem}
\begin{Rem}
We understand Hypotheses \ref{hypmon} in the sense that $f(\cdot,p)$ satisfies the monotonicity assumptions for any value of $p$ and, as previously, developments such as in Section \ref{sec:extensionuniqueness} could have been carried on.
\end{Rem}
\begin{proof}
The proof of this result is once again very similar to the one of Theorem \ref{uniqtime}. Let us take two solutions $U$ and $V$ and define $W :  [0,T]^2 \times \mptd^2\times \mathbb{T}^k$ by
\begin{equation}
W(t,s,µ,\nu,p) = \langle U(t,\cdot,µ,p) - V(s,\cdot,\nu,p), µ - \nu \rangle.
\end{equation}
Let us now remark, following the same argument as in the proof of Theorem \ref{uniqtime}, that at a point $(t_0,s_0,µ_0,\nu_0,p_0)$ of strict minimum of $W$ (up to the addition small perturbations using Lemma \ref{Stegall}), the term arising in the relation of monotone solutions from the additional terms in $p$ is of the form
\begin{equation}
-b(p_0)\cdot \nabla_p W(t_0,s_0,µ_0,\nu_0,p_0) - \sigma' \Delta_pW(t_0,s_0,µ_0,\nu_0,p_0).
\end{equation}
Because we are at a point of minimum of $W$ (in particular it is also a minimum in $p$), this term has a sign and the rest of the proof follows quite easily.
\end{proof}
\begin{Rem}
Because few information is needed for the regularity in $p$ of $W$ in the previous proof, it is very likely that continuity with respect to $p$ and viscosity solution like information are sufficient to characterize monotone solution of \eqref{mecp}, although we do not claim that such results are trivially in the scope of this paper.
\end{Rem}

\subsection{Common jumps}\label{sec:commonjumps}
We now introduce, in the continuous state space framework, a type of common noise similar to the one introduced in \citep{bertucci2019some}. More precisely, we want to model situations in which, at random times which are given by a Poisson process of intensity $\lambda > 0$, all the players in the game are affected by a common transformation. This transformation can be deterministic, for instance all the players in the state $x$ are transported in a state $\Lambda (x)$. It can also carry a form of randomness which is distributed in an i.i.d. fashion among the players. In such a situation, all the players in a state $x$ are going to be transported to a new state which is drawn according to a distribution on the state space $K(\cdot,x)$, independently from one another. So that if before the jumps, the players are distributed according to $m \in \mptd$, they are distributed according to $\int_{\mathbb{T}^d}K(x,y)m(dy)$ immediately after the jump. We refer to \citep{bertucci2019some} for more details on this type of noise (in the finite state space case).\\

In the following we assume that $K$ is a non-negative smooth function on $(\mathbb{T}^d)^2$ such that for all $y \in \mathbb{T}^d$, $\int_{\mathbb{T}^d}K(x,y)dx = 1$. We define the operator $\mathcal{T}$ by
\begin{equation}\label{defT}
\forall m \in \mptd, x \in \mathbb{T}^d, \mathcal{T}(m)(x) = \int_{\mathbb{T}^d}K(x,y)m(dy).
\end{equation}
Let us recall that the adjoint $\mathcal{T}^*$ of $\mathcal{T}$ is given by
\begin{equation}
\forall \phi \in \mathcal{C}^0(\mathbb{T}^d), y \in \mathbb{T}^d, \mathcal{T}^*(\phi)(y) = \int_{\mathbb{T}^d}K(x,y)\phi(x)dx.
\end{equation}

Because the players anticipate the noise and the fact that they are going to be transported to another state, the associated master equation is given by
\begin{equation}\label{met}
\begin{aligned}
\partial_t& U - \sigma \Delta U + H(x,\nabla_x U) - \left\langle \dmU(x,m,\cdot), \text{div}\left(D_pH(\cdot, \nabla_x U(\cdot,m))m\right)\right\rangle\\
&+ \lambda \bigg(U - \mathcal{T}^*(U(t,\cdot,\mathcal{T}(m)))\bigg) - \sigma\left\langle \dmU(x,m,\cdot), \Delta m \right\rangle = f(x,m) \text{ in } (0,\infty)\times\mathbb{T}^d\times \mptd,\\
&U(0,x,m) = U_0(x,m) \text{ in } \mathbb{T}^d\times \mptd.
\end{aligned}
\end{equation}
This master equation is obviously reminiscent of the one studied in \citep{bertucci2020monotone} (in the finite state space case). Once again we state an appropriate notion of solution for this equation.
\begin{Def}\label{deftt}
A function $U \in \mathcal{B}_t$ is a monotone solution of \eqref{met}  if \begin{itemize} \item For any $\mathcal{C}^2$ function $\phi : \mathbb{T}^d \to \mathbb{R}$, for any measure $\nu \in \mmtd$, for any smooth function $\vartheta : [0,\infty) \to \mathbb{R}$, any $T>0$ and any point $(t_0,m_0) \in (0,T]\times\mptd$ of strict minimum of $(t,m) \to \langle U(t,\cdot,m)- \phi, m- \nu \rangle - \vartheta(t)$ on $[0,T]\times \mptd$, the following holds
\begin{equation}
\begin{aligned}
\frac{d \vartheta}{dt}(t_0)& +  \left\langle - \sigma \Delta U + H(\cdot,\nabla_x U) + \lambda \bigg( U - \mathcal{T}^*(U(t,\cdot,\mathcal{T}(m)))\bigg) , m_0 - \nu \right\rangle\\
 \geq& \langle f(\cdot,m_0),m_0 - \nu\rangle-  \langle U - \phi, \text{div}(D_pH(\nabla_x U)m_0)\rangle  - \sigma \langle \Delta(U- \phi),m_0\rangle.
\end{aligned}
\end{equation}
\item The initial condition holds
\begin{equation}
U(0,\cdot,\cdot) = U_0(\cdot,\cdot).
\end{equation}
\end{itemize}
\end{Def}
As we did in the previous cases, we can establish the following uniqueness result. 
\begin{Theorem}
Under Hypothesis \ref{hypmon}, two monotone solutions of \eqref{met} in the sense of Definition \ref{deftt} only differ by a function $c :[0,\infty) \times\mptd \to \mathbb{R}$. If such a monotone solution $U$ exists, then $U(t)$ is actually monotone for all time $t \geq 0$.
\end{Theorem}
\begin{proof}
The proof of this result is once again very similar to the one of Theorem \ref{uniqtime}. For the first part of the claim, let us take two solutions $U$ and $V$ and define $W :  [0,T]^2 \times \mptd^2 : \mathbb{R}$ by
\begin{equation}
W(t,s,µ,\nu) = \langle U(t,\cdot,µ) - V(s,\cdot,\nu), µ - \nu \rangle.
\end{equation}
Let us now remark, following the same argument as in the proof of Theorem \ref{uniqtime}, that at a point $(t_0,s_0,µ_0,\nu_0)$ of strict minimum of $W$ (up to the addition small perturbations using Lemma \ref{Stegall}), the relation of monotone solutions one obtains is similar to the classical one except  for the addition of a term in $\lambda$. When combining the relation from $U$ and the one from $V$, one obtain the same relation except for the addition of the term
\begin{equation}
\lambda\bigg(W(t_0,s_0,\mu_0,\nu_0) - W(t_0,s_0,\mathcal{T}(\mu_0),\mathcal{T}(\nu_0))\bigg).
\end{equation}
Because we are at a point of minimum of $W$, this term has a sign and the rest of the proof follows quite easily. 
\end{proof}

\subsection{An a priori estimate for the solution of the master equation with common noise}\label{subs:apriori}
We now show an a priori estimate which is essentially valid for all the master equations we have written up to now. We only state (and prove) it in the case of \eqref{met} and we leave its generalization to the other master equations to the interested reader. Although it is not sufficient to establish general results of existence, we believe that it may be a good starting point for such results. Furthermore, it is the essential tool to prove existence of solutions of \eqref{met} (the forthcoming Theorem \ref{thm:existt}). \\

To state our result in an understandable fashion while taking care of the normalization constraint used in the definition of $\dmU$, we work with a different normalization of $\dmU$ for the remainder of this Section. Namely we now impose that $\dmU$ is such that $\int_{\mathbb{T}^d}\dmU = 0$, taking this normalization against the Lebesgue measure instead of the measure $m$ itself.\\

The a priori estimate we want to present is valid only under additional monotonicity assumptions on $f$ and $U_0$. We assume that $f$ and $U_0$ are differentiable with respect to the measure argument and that there exists $\alpha > 0$ such that for all $µ \in \mptd, \nu \in L^2_0(\mathbb{T}^d):= \{f \in L^2, \int f = 0\},$
\begin{equation}\label{strongmon1}
\left\langle \nu \bigg| \frac{\delta f}{\delta m}(\cdot,µ,\cdot) \bigg| \nu  \right \rangle \geq \alpha \left\| \left \langle \frac{\delta f}{\delta m}(\cdot,µ), \nu\right\rangle \right\|_{L^2}^2,
\end{equation}
\begin{equation}\label{strongmon2}
\left\langle \nu \bigg| \frac{\delta U_0}{\delta m}(\cdot,µ,\cdot) \bigg| \nu  \right \rangle \geq \alpha \left\| \left \langle \frac{\delta U_0}{\delta m}(\cdot,µ), \nu\right\rangle \right\|_{L^2}^2.
\end{equation}
We only consider $\nu \in L^2_0$ because of the fact that $f$ and $U_0$ are only defined on $\mptd$ and not on all measures. The previous assumption is a sort of strong monotonicity assumption on $f$ and $U_0$. For instance for $\alpha = 0$ this assumption reduces to usual monotonicity for smooth $f$ and $U_0$. Furthermore, this assumption is weaker than $\alpha$ monotonicity. Indeed, if $f$ is smooth and satisfies for all $\nu \in L^2_0(\mathbb{T}^d)$,
\begin{equation}
\left\langle \nu \bigg| \frac{\delta f}{\delta m}(\cdot,µ,\cdot) \bigg| \nu  \right \rangle \geq \alpha \left\| \nu \right\|_{L^2}^2,
\end{equation}
for some $\alpha > 0$, then it satisfies \eqref{strongmon1} for all $\nu \in L^2(\mathbb{T}^d)$ (possibly for another $\alpha > 0$). Let us remark that such a requirement is satisfied for functions $f$ defined by 
\begin{equation}
\forall x \in \mathbb{T}^d, m \in \mptd, f(x,m) = \int_{\mathbb{T}^d}\Psi(z,m\star\rho(z))\rho(x-z)dz,
\end{equation}
for a smooth non negative function $\rho$ and a smooth function $\Psi : \mathbb{T}^d\times \mathbb{R} \to \mathbb{R}$ whose derivative with respect to the second argument $\partial_y \Psi$ satisfies $C^{-1} \leq \partial_y \Psi \leq C$ for some constant $C>0$. We can now state the a priori estimate.
\begin{Prop}\label{prop:est}
Under Hypothesis \ref{hypmon}, the assumption that $\mathcal{T}$ is continuous and linear and the stronger requirements \eqref{strongmon1} and \eqref{strongmon2}, for any $t_f > 0$, there exists $C > 0$ such that a classical solution $U$ of \eqref{met} satisfies for $t \in (0,t_f),\nu, \nu' \in L^2(\mathbb{T}^d)$ : 
\begin{equation}\label{aprioriest}
\left| \left \langle \nu \left |\dmU(t,\cdot,µ,\cdot)\right| \nu'\right \rangle \right | \leq C \|\nu\|_{L^2}\| \nu'\|_{L^2},
\end{equation}
where $C$ depends only on $\alpha, \sigma,H, \lambda$ and $\mathcal{T}$. If $\mathcal{T}$ is non expansive in $L^2$, then $C$ only depends on $\alpha,\sigma$ and $H$.
\end{Prop}
\begin{proof}
Let us define $W,Z_{\beta} : [0,t_f]\times \mptd \times L^2_0(\mathbb{T}^d) \to \mathbb{R}$ by
\begin{equation}
W(t,µ,\nu) = \langle U(t,\cdot,µ), \nu\rangle,
\end{equation}
\begin{equation}
Z_{\beta}(t,µ,\nu) = \left\langle \frac{\delta W}{\delta µ}(t,µ,\nu),\nu\right\rangle - \beta(t) \left\langle  \frac{\delta W}{\delta µ}(t,µ,\nu), \frac{\delta W}{\delta µ}(t,µ,\nu)\right\rangle,
\end{equation}
for $\beta : [0,\infty) \to \mathbb{R}$ to be defined later on. Let us remark that $Z_{\beta}$ is a quadratic and smooth function of $\nu$. We denote by $\frac{\delta Z_{\beta}}{\delta \nu}$ the gradient of $Z_{\beta}$ in $\nu \in L^2_0(\mathbb{T}^d)$ (since $L^2_0(\mathbb{T}^d)$ is an Hilbert space, this is a usual gradient). The chain rule yields that $Z_{\beta}$ is a solution on $(0,\infty)\times \mptd \times L^2_0(\mathbb{T}^d)$ of 
\begin{equation}
\begin{aligned}
&\partial_t Z_{\beta} + \left \langle - \text{div}\left(  D_{pp}H(\cdot, \nabla_x U(\cdot,µ))\nabla_x  \frac{\delta Z_{\beta}}{\delta \nu}(t,µ,\nu,\cdot)  µ \right) , \frac{\delta W}{\delta µ}(t,µ,\nu)  \right \rangle  \\
&+ \left \langle - \text{div}\left(D_pH(\cdot, \nabla U(\cdot,µ))µ\right) - \sigma \Delta µ , \frac{\delta Z_{\beta}}{\delta µ}\right \rangle + \lambda \left(   Z_{\beta} - Z_{\beta}(t,\mathcal{T}µ,\mathcal{T}\nu)   \right)\\
& - \left \langle \Delta \frac{\delta Z_{\beta}}{\delta \nu}(t,µ,\nu,\cdot) - D_pH(\cdot,\nabla U)\nabla  \frac{\delta Z_{\beta}}{\delta \nu}(t,µ,\nu,\cdot), \nu\right \rangle\\
=&\left\langle \nu \bigg| \frac{\delta f}{\delta m}(\cdot,µ,\cdot) \bigg| \nu  \right \rangle + \left \langle - \text{div}\left(  D_{pp}H(\cdot, \nabla_x U(\cdot,µ))\nabla_x   \frac{\delta W}{\delta µ}(t,µ,\nu,\cdot)  µ \right) , \frac{\delta W}{\delta µ}(t,µ,\nu)  \right \rangle \\
&- \left \langle - \sigma \Delta \frac{\delta W}{\delta µ} - \text{div}\left(D_pH(\cdot,\nabla U)\frac{\delta W}{\delta µ}\right),\nu    \right \rangle\\
& - \left \langle\sigma \Delta  \frac{\delta W}{\delta µ}(t,µ,\nu,\cdot) - D_pH(\cdot,\nabla U)\nabla   \frac{\delta W}{\delta µ}(t,µ,\nu,\cdot), \nu\right \rangle\\
&-2\beta \left \langle\nu \bigg| \frac{\delta f}{\delta m}(\cdot,µ,\cdot) \bigg| \frac{\delta W}{\delta µ} \right \rangle + 2\beta \left \langle - \sigma \Delta \frac{\delta W}{\delta µ} - \text{div}\left(D_pH(\cdot,\nabla U)\frac{\delta W}{\delta µ}\right),\frac{\delta W}{\delta µ}    \right \rangle\\
& + \beta \lambda \left( \left\| \frac{\delta W}{\delta µ} \right\|^2 + \left\| \frac{\delta W}{\delta µ} (t,\mathcal{T}µ,\mathcal{T}\nu)\right\|^2 - 2 \left \langle \mathcal{T}\frac{\delta W}{\delta µ} , \frac{\delta W}{\delta µ} (t,\mathcal{T}µ,\mathcal{T}\nu)\right\rangle \right) - \frac{d \beta}{dt}  \left\| \frac{\delta W}{\delta µ} \right\|^2.
\end{aligned}
\end{equation}
The convexity of $H$ and calculus on the term in $\lambda$ yields
\begin{equation}\label{eqzbeta}
\begin{aligned}
&\partial_t Z_{\beta} + \left \langle - \text{div}\left(  D_{pp}H(\cdot, \nabla_x U(\cdot,µ))\nabla_x  \frac{\delta Z_{\beta}}{\delta \nu}(t,µ,\nu,\cdot)  µ \right) , \frac{\delta W}{\delta µ}(t,µ,\nu)  \right \rangle  \\
&+ \left \langle - \text{div}\left(D_pH(\cdot, \nabla U(\cdot,µ))µ\right) - \sigma \Delta µ , \frac{\delta Z_{\beta}}{\delta µ}\right \rangle + \lambda \left(   Z_{\beta} - Z_{\beta}(t,\mathcal{T}µ,\mathcal{T}\nu)   \right)\\
& - \left \langle \Delta \frac{\delta Z_{\beta}}{\delta \nu}(t,µ,\nu,\cdot) - D_pH(\cdot,\nabla U)\nabla  \frac{\delta Z_{\beta}}{\delta \nu}(t,µ,\nu,\cdot), \nu\right \rangle\\
\geq&\left\langle \nu \bigg| \frac{\delta f}{\delta m}(\cdot,µ,\cdot) \bigg| \nu  \right \rangle -2\beta \left \langle\nu \bigg| \frac{\delta f}{\delta m}(\cdot,µ,\cdot) \bigg| \frac{\delta W}{\delta µ} \right \rangle  + \beta \lambda \left( \left\| \frac{\delta W}{\delta µ} \right\|^2 - \left\| \mathcal{T}\frac{\delta W}{\delta µ}\right\|^2  \right)  \\
&+ 2\beta \left \langle - \sigma \Delta \frac{\delta W}{\delta µ} - \text{div}\left(D_pH(\cdot,\nabla U)\frac{\delta W}{\delta µ}\right),\frac{\delta W}{\delta µ}    \right \rangle - \frac{d \beta}{dt}  \left\| \frac{\delta W}{\delta µ} \right\|^2.
\end{aligned}
\end{equation}
Remark that
\begin{equation}
\left \langle - \sigma \Delta \frac{\delta W}{\delta µ} - \text{div}\left(D_pH(\cdot,\nabla U)\frac{\delta W}{\delta µ}\right),\frac{\delta W}{\delta µ}    \right \rangle = \sigma \left\| \nabla_x \frac{\delta W}{\delta µ} \right\|^2 + \left\langle D_pH(\cdot,\nabla_x U)\frac{\delta W}{\delta µ},\nabla_x \frac{\delta W}{\delta µ} \right\rangle
\end{equation}
Hence, using the assumption on $f$, we deduce that 
\begin{equation}
\begin{aligned}
&\left\langle \nu \bigg| \frac{\delta f}{\delta m}(\cdot,µ,\cdot) \bigg| \nu  \right \rangle -2\beta \left \langle\nu \bigg| \frac{\delta f}{\delta m}(\cdot,µ,\cdot) \bigg| \frac{\delta W}{\delta µ} \right \rangle  + \beta \lambda \left( \left\| \frac{\delta W}{\delta µ} \right\|^2 - \left\| \mathcal{T}\frac{\delta W}{\delta µ}\right\|^2  \right) - \frac{d \beta}{dt}  \left\| \frac{\delta W}{\delta µ} \right\|^2\\
& + \left \langle - \sigma \Delta \frac{\delta W}{\delta µ} - \text{div}\left(D_pH(\cdot,\nabla U)\frac{\delta W}{\delta µ}\right),\frac{\delta W}{\delta µ}    \right \rangle\\
&\geq \left(-\alpha^{-2}\beta^2 - \frac{\|D_pH\|_\infty}{2\sigma}\beta+ \lambda\left( 1 - \left\|\mathcal{T}\right\|_{\mathcal{L}(L^2)}\right)\beta - \frac{d\beta}{dt}\right) \left\| \frac{\delta W}{\delta µ} \right\|^2.
\end{aligned}
\end{equation}
From the assumption on $U_0$, we deduce that for any $\beta(0) \in (0,\alpha)$, $Z_{\beta}$ is non-negative at $t = 0$. Hence we deduce that there exists $\beta$ such that \begin{itemize}
\item $\beta$ is defined on $[0,t_f]$.
\item The right hand side of \eqref{eqzbeta} is non negative for $t \in (0,t_f)$
\item  $Z_{\beta}$ is non-negative at $t = 0$
\item There exists $C > 0$ depending only on $\alpha, \lambda,\sigma, H$ and $\mathcal{T}$ such that for $t\in (0,t_f)$, $\beta(t) \geq C^{-1}$.
\end{itemize}
We shall now make use of a comparison principle to prove that $Z_{\beta}$ is non negative for $t \in [0,t_f]$. Assume it is not the case, then there exists $\kappa > 0$, such that there exists, $t \in (0,t_f], µ \in \mptd, \nu \in L^2_0(\mathbb{T}^d)$ such that
\begin{equation}
Z_{\beta}(t,µ,\nu) \leq -\kappa t < 0.
\end{equation}
Consider now $t^*$ defined as
\begin{equation}
t^* = \text{argmin} \{t \in (0,t_f], \exists µ \in \mptd, \nu \in L^2_0(\mathbb{T}^d), Z_{\beta}(t,µ,\nu) < -\kappa t\}.
\end{equation}
Since $\mptd$ is compact, and $Z_{\beta}$ is quadratic in $\nu$, there exists $µ^* \in \mptd, \nu^* \in L^2_0(\mathbb{T}^d)$ such that 
\begin{equation}
\begin{cases}
Z_{\beta}(t^*,µ^*,\nu^*) \leq -\kappa t^*,\\
\partial_t Z_{\beta}(t^*,µ^*,\nu^*) \leq -\kappa,\\
 \frac{\delta Z_{\beta}}{\delta \nu}(t^*,µ^*,\nu^*) = 0,\\
 µ \to Z_{\beta}(t^*,µ,\nu^*) \text{ reaches its minimum at } µ^*.
\end{cases}
\end{equation}
Hence evaluating \eqref{eqzbeta} at $(t^*,µ^*,\nu^*)$ yields
\begin{equation}
-\kappa \geq 0,
\end{equation}
which is a contradiction, thus $Z_{\beta} \geq 0$.\\

Rewriting the non-negativity of $Z_{\beta}$, we obtain for any $t \in (0,t_f), \nu \in L^2_0, µ \in \mptd$
\begin{equation}
\left\| \left \langle \dmU(\cdot,µ), \nu\right \rangle \right \|_{L^2}^2 \leq C \left\langle \nu \left| \dmU(t,\cdot,µ,\cdot)\right|\nu\right\rangle.
\end{equation}
From Cauchy-Schwarz inequality we deduce that 
\begin{equation}
\left\| \left \langle \dmU(\cdot,µ), \nu\right \rangle \right \|_{L^2} \leq C \left\| \nu \right\|_{L^2}.
\end{equation}
The \color{black}required \color{black}estimate then easily follows. The fact that, when $\mathcal{T}$ is non expansive for the $L^2$ norm $C$ does not depend on $\lambda$ nor $\mathcal{T}$, can be easily observed at the level of \eqref{eqzbeta}.
\end{proof}
\begin{Rem}
Let us insist on the fact that the a priori estimate we just presented is also valid for the master equations \eqref{me2p} and \eqref{mecp}, the proofs of these facts follow exactly the same argument as the one we presented.
\end{Rem}

\subsection{Existence of solutions of first order master equations with common noise}
As we already said, the question of existence of solutions of the previous master equations is not the main interest of this paper. However, in this section, we take some time to explain how the previous a priori estimate is helpful to establish the existence of a monotone solution of \eqref{met} under some additional assumptions on the jump operator $\mathcal{T}$. We also explain briefly afterwards how such a result can be generalized.

\begin{Theorem}\label{thm:existt}
Assume that $\mathcal{T}$ is given by \eqref{defT} for some smooth non negative function $K$ and that the assumption of Theorem \ref{existclassic} and Proposition \ref{prop:est} are satisfied. Then there exists a (unique) monotone solution of \eqref{met} in the sense of definition \ref{deftt}.
\end{Theorem}
\begin{proof}
Let us define the operator $\mathcal{A}$ which associates to a function $V$ the solution $U:= \mathcal{A}(V)$ of the master equation
\begin{equation}
\begin{aligned}
\partial_t& U - \sigma \Delta U + H(x,\nabla_x U) - \left\langle \dmU(x,m,\cdot), \text{div}\left(D_pH(\cdot, \nabla U(\cdot,m))m\right)\right\rangle\\
&+ \lambda U  - \sigma\left\langle \dmU(x,m,\cdot), \Delta m \right\rangle = f(x,m) + \lambda \mathcal{T}^*(V(t,\cdot,\mathcal{T}(m)))\text{ in } (0,\infty)\times\mathbb{T}^d\times \mptd,\\
&U(0,x,m) = U_0(x,m) \text{ in } \mathbb{T}^d\times \mptd.
\end{aligned}
\end{equation}
Considering $\mathcal{T}^*(V(t,\cdot,\mathcal{T}(m)))$ in the same way as $f$, the previous equation falls in the scope of the previous analysis (Theorem \ref{existclassic} and Proposition \ref{existmon}) except for the presence of the term $\lambda U$ which does not play any sort of role in the previous result. Thus, we admit that they extend to this situation.

From this we deduce that the operator $\mathcal{A}$ is well defined from the set of functions $V$ which are monotone and satisfies the assumptions (for $f$) of Theorem \ref{existclassic}. Moreover, if $V(t)$ is $C$ Lipschitz, uniformly in $t$, seen as an operator from $L^2(\mathbb{T}^d)$ into itself, then from the regularity of the kernel $K$ that we assumed (the kernel of the operator $\mathcal{T}$), then $\mathcal{T}^*(V(t,\cdot,\mathcal{T}(\cdot)))$ is in fact (uniformly in $t$) Lipschitz continuous from $\mptd$ to $\mathcal{C}^2(\mathbb{T}^d)$. From this we deduce that the operator $\mathcal{A}$ is compact (Proposition \ref{existmon}).

Moreover, the a priori estimate of Proposition \ref{prop:est} yields that $\mathcal{E} := \{U, \exists \theta \in [0,1], \theta \mathcal{A}(U) = U\}$ is bounded in the set of Lipschitz continuous functions $[0,t_f]\times L^2(\mathbb{T}^d) \to L^2(\mathbb{T}^d)$ for any $t_f > 0$. Hence, from the assumptions we made on $\mathcal{T}$, we know that $\mathcal{T}(\mathcal{E})$ is such that it uniformly satisfies the assumptions on $f$ of Proposition \ref{holderU}. Thus, thanks to Proposition \ref{holderU}, it is a uniformly continuous family of functions over $[0,t_f]\times \mathbb{T}^d\times \mptd$. Hence it is bounded in $\mathcal{B}_t$ (since the uniform regularity of order $2$ in the space variable $x$ is also true). Finally let us recall that from Proposition \ref{stabtime}, we know that $\mathcal{A}$ is continuous for the topology of $\mathcal{B}_t$. Hence we deduce that $\mathcal{A}$ has a fixed point.
\end{proof}
Let us comment on extensions of this result. First, concerning master equations involving additional stochastic parameters, it is clear that the case of \eqref{me2p} can be treated in a similar fashion. The case of \eqref{mecp} requires an additional a priori estimate on the regularity of the solution with respect to $p$.

Moreover, it also seems possible to extend the previous result to more general jump operators $\mathcal{T}$. We now briefly expand on how this should be possible. Let us recall that the a priori estimate of Proposition \ref{prop:est} only requires the $L^2$ Lipschitz continuity of $\mathcal{T}$ whereas the assumptions of Theorem \ref{thm:existt} require strong regularization properties of $\mathcal{T}$. Let us now recall that obviously the presence of the i.i.d. noise in the MFG induces a regularization properties of the underlying Fokker-Planck equation. In other words, if at a time $t$ the repartition of players is characterized by a measure $m(t)$ and if all the players play the Nash equilibrium, then instantly the measure describing the repartition of players has a density in $L^2(\mathbb{T}^d)$ (even much more regular than $L^2$ under the standing assumptions in fact). Then we could use the $L^2$ a priori estimate of Proposition \ref{prop:est} to gain additional regularity for solutions of \eqref{met} with bootstrapping techniques and ideas from Proposition \ref{holderU} (or more precisely from Section 3 in \citep{cardaliaguet2019master}).

\subsection{Asymptotic differential operators from common jumps}
Another interesting features of the kind of common noise of Section \ref{sec:commonjumps} is that one can obtained differential terms in the $m$ variable by taking limits of such noises. We briefly extend the results and ideas from \citep{bertucci2019some} to this continuous state space case. Let us first mention, that the particular form of the operator $\mathcal{T}$ was little used in the previous section. We mainly used its linearity (more precisely the existence of an adjoint...). Hence, we shall only assume here that $\mathcal{T}$ is a linear operator from the space of measures on $\mathbb{T}^d$ into itself which maps $\mptd$ into itself. Let us remark that we do not establish precise result of convergence in this section but rather present the natural asymptotics one can obtain from common jumps.\\

\subsubsection{General case}
In a first time, let us assume that $\mathcal{T}$ is of the form
\begin{equation}
\mathcal{T}= Id + \lambda^{-1} \mathcal{S},
\end{equation}
for a linear operator $\mathcal{S}$ which maps $\mptd$ into the set of measure of mass zero. The previous study can be extended for such an operator $\mathcal{T}$, in particular the associated master equation is still \eqref{met}. Observe that, at least formally,
\begin{equation}
\lambda (U(t,x,m) - \mathcal{T}^*U(t,x,\mathcal{T}(m))) \underset{\lambda \to \infty}{\longrightarrow} - \left\langle \dmU(t,x,m,\cdot), \mathcal{S}(m)\right\rangle - \mathcal{S}^*(U(t,x,m)).
\end{equation}
As one can guess, the limit master equation 
\begin{equation}\label{metlim1}
\begin{aligned}
\partial_t& U - \sigma \Delta U + H(x,\nabla_x U) - \left\langle \dmU(x,m,\cdot), \text{div}\left(D_pH(\cdot, \nabla U(\cdot,m))m\right)\right\rangle - \mathcal{S}^*(U(t,x,m)))\\
&- \left\langle \dmU(t,x,m,\cdot), \mathcal{S}(m)\right\rangle  - \sigma\left\langle \dmU(x,m,\cdot), \Delta m \right\rangle = f(x,m) \text{ in } (0,\infty)\times\mathbb{T}^d\times \mptd,\\
&U(0,x,m) = U_0(x,m) \text{ in } \mathbb{T}^d\times \mptd,
\end{aligned}
\end{equation}
also propagates monotonicity and is thus adequate for a notion of monotone solution which can be stated in this situation with 
\begin{Def}
A function $U \in \mathcal{B}_t$ is a monotone solution of \eqref{met}  if \begin{itemize} \item For any $\mathcal{C}^2$ function $\phi : \mathbb{T}^d \to \mathbb{R}$, for any measure $\nu \in \mmtd$, for any smooth function $\vartheta : [0,\infty) \to \mathbb{R}$, any $T > 0$ and any point $(t_0,m_0) \in (0,T]\times\mptd$ of strict minimum of $(t,m) \to \langle U(t,\cdot,m)- \phi, m- \nu \rangle - \vartheta(t)$ on $[0,T]\times \mptd$, the following holds
\begin{equation}
\begin{aligned}
\frac{d \vartheta}{dt}(t_0)& +  \langle - \sigma \Delta U + H(\cdot,\nabla_x U) , m_0 - \nu \rangle - \langle \mathcal{S}^*\phi, m_0\rangle + \langle \mathcal{S}^*U, \nu\rangle\\
 \geq& \langle f(\cdot,m_0),m_0 - \nu\rangle-  \langle U - \phi, \text{div}(D_pH(\nabla_x U)m_0)\rangle  - \sigma \langle \Delta(U- \phi),m_0\rangle.
\end{aligned}
\end{equation}
\item The initial condition holds
\begin{equation}
U(0,\cdot,\cdot) = U_0(\cdot,\cdot).
\end{equation}
\end{itemize}
\end{Def}
We do not detail another uniqueness result for such a case. Following \citep{bertucci2019some}, we could also obtained second order terms in a similar fashion. Because we detail such a fact on a specific example below, we do not focus on this asymptotic right now.

\subsubsection{Asymptotic terms associated to translations}\label{sec:translations}
An important case of common jumps is the one in which $\mathcal{T}(m)$ is the image measure of $m$ by some application $B : \mathbb{T}^d \to \mathbb{T}^d$. In such a case, at the random times at which the jumps occur, all the players in $x$ are transported to $B(x)$. Let us assume that $B$ is of the form $ B = Id + \lambda^{-1}\tilde{B}$ for some smooth $\tilde{B} : \mathbb{T}^d \to \mathbb{T}^d$. In this context, formally, one obtain that 
\begin{equation}
\lambda (U(t,x,m) - \mathcal{T}^*U(t,x,\mathcal{T}(m))) \underset{\lambda \to \infty}{\longrightarrow} - \langle D_mU(t,x,m,\cdot)\tilde{B}(\cdot), m\rangle - \tilde{B}(x) \cdot \nabla_x U(t,x,m).
\end{equation}
Following \citep{bertucci2019some}, the addition of several terms modeling common jumps does not raise any particular issue and can be treated in a similar way. By doing so, we are able to obtain higher order asymptotic terms in the master equation as we now explain. Let us define $\mathcal{T}_+(m)$ the image measure of $m$ by $ Id + \lambda^{-1/2}\tilde{B}$ and $\mathcal{T}_-(m)$ the image measure of $m$ by $ Id - \lambda^{-1/2}\tilde{B}$. In such a situation, we can remark the following asymptotic
\begin{equation}\label{asymptotic2}
\begin{aligned}
\lambda(2U(t,x,m) &- \mathcal{T}^*_+U(t,x,\mathcal{T}_+(m)) - \mathcal{T}^*_-U(t,x,\mathcal{T}_-(m)))\\
 \underset{\lambda \to \infty}{\longrightarrow}& - \tilde{B}(x) \cdot D^2_{xx} U(t,x,m)\cdot \tilde{B}(x) - 2 \tilde{B}(x)\cdot \nabla_x \langle \tilde{B}(\cdot)D_m U(t,x,m,\cdot),m\rangle\\
&- \int_{\mathbb{T}^d} \int_{\mathbb{T}^d} \tilde{B}(y)D^2_{mm}U(t,x,m,y,z)\tilde{B}(z)m(dy)m(dz)\\
 &- \langle \tilde{B}(\cdot)D^2_{yy}\dmU(t,x,m,\cdot)\tilde{B}(\cdot),m\rangle,
\end{aligned}
\end{equation}
where $D^2_{mm} := D_m(D_mU)$.
This gives us another way to derive the master equation with common noise studied in \citep{cardaliaguet2019master}, which is written below as equation \eqref{me2}. Indeed to recover \eqref{me2}, one has to sum such terms as in \eqref{asymptotic2} for the constant maps $\tilde{B}_i(y) = (0,...,0,\sqrt{\beta}y_i,0,...0)$ for $i = 1$ to $i= d$. Even though the proper study of such an asymptotic is not the subject of this paper (and thus we do not it here), we believe that such an approach can be insightful in many ways for the study of \eqref{me2}. Moreover this justifies in some sense the generality of the noise studied in Section \ref{sec:commonjumps}.

\section{Master equations of second order}\label{sec:secondorder}
We now turn to the case of second order master equations such as the one studied in \citep{cardaliaguet2019master}. We are here interested in the equation
\begin{equation}\label{me2}
\begin{aligned}
\partial_t U& - (\sigma  + \beta)\Delta U + H(x,\nabla_x U) - \left\langle \dmU(x,m,\cdot), \text{div}\left(D_pH(\cdot, \nabla U(\cdot,m))m\right)\right\rangle\\
& - (\sigma+ \beta)\left\langle \dmU(x,m,\cdot), \Delta m \right\rangle  + 2\beta \nabla_x \cdot \left\langle \dmU(x,m,\cdot),\nabla m\right\rangle\\
&- \beta \left \langle \nabla m \bigg| \frac{\delta^2 U}{\delta m ^2}(x,m,\cdot,\cdot) \bigg|\nabla m\right\rangle = f(x,m),\\
U(0&,x,m) = U_0(x,m).
\end{aligned}
\end{equation}
Let us precise that the last term of the left hand side of \eqref{me2} is understood as $-\beta \sum_{i=1}^d \langle \partial_i m| \frac{\delta^2 U}{\delta m^2}|\partial_i m\rangle.$ As we did for \eqref{me}, we now state a uniqueness result for \eqref{me2} for which we give a different proof than in \citep{cardaliaguet2019master}, where it originates from. 
\begin{Prop}
Under Hypothesis \ref{hypmon}, there exists at most one classical solution of \eqref{me2}.
\end{Prop}
\begin{proof}
Let us consider $U$ and $V$ two smooth solutions and define $W$ by
\begin{equation}
W(t,µ,\nu) = \langle U(t,\cdot,\mu) - V(t,\cdot, \nu), µ - \nu\rangle.
\end{equation}
A simple computation yields that $W$ is a solution of
\begin{equation}
\begin{aligned}
\partial_t &W-  \langle (\sigma + \beta)\Delta_x (U-V) + H(\cdot,\nabla U) - H(\cdot,\nabla V), µ- \nu\rangle \\
&- \left\langle \frac{\delta W}{\delta µ}(µ,\nu), \text{div}\left(D_pH(\cdot, \nabla U(\cdot,µ))µ\right)\right\rangle - \left\langle \frac{\delta W}{\delta \nu}(µ,\nu), \text{div}\left(D_pH(\cdot, \nabla V(\cdot,\nu))\nu\right)\right\rangle\\
 &- (\sigma + \beta)\left\langle \frac{\delta W}{\delta µ}(µ,\nu), \Delta µ \right\rangle - (\sigma + \beta)\left\langle \frac{\delta W}{\delta \nu}(µ,\nu), \Delta \nu \right\rangle\\
&-\beta\left(    \left \langle \nabla µ \bigg| \frac{\delta^2 W}{\delta µ ^2}(µ,\nu,\cdot,\cdot) \bigg|\nabla µ\right\rangle +   \left \langle \nabla \nu \bigg| \frac{\delta^2 W}{\delta \nu ^2}(µ,\nu,\cdot,\cdot) \bigg|\nabla \nu\right\rangle - 2 \left \langle \nabla \nu \bigg| \frac{\delta^2 W}{\delta µ \delta \nu}(µ,\nu,\cdot,\cdot) \bigg|\nabla \mu\right\rangle \right)\\
=& \langle f(\cdot,µ) - f (\cdot,\nu),µ - \nu\rangle - \left\langle U(\cdot,µ) - V(\cdot,\nu), \text{div}\left(D_pH(\cdot, \nabla U(\cdot,µ))µ\right)\right\rangle\\
&- \left\langle V(\cdot,\nu) - U(\cdot,µ), \text{div}\left(D_pH(\cdot, \nabla V(\cdot,\nu))\nu\right)\right\rangle - (\sigma + \beta) \langle \Delta_x (U-V), µ- \nu\rangle.
\end{aligned}
\end{equation}
Once again, using the monotonicity of $f$ and the convexity of $H$, we obtain that the right hand side of the previous equation is non-negative. We now show how a comparison principle can extend this non-negativity of the right hand side to $W$. Assume that there exists $t>0,µ,\nu$ such that $W(t,µ,\nu) < 0$. Under this assumption, there exists $\kappa > 0$ such that $W(t,µ,\nu) + \kappa t < 0$. Consider now $t^*$ defined as
\begin{equation}
t^* = \text{argmin} \{t > 0, \exists µ,\nu \in \mptd, W(t,µ,\nu) + \kappa t < 0\},
\end{equation}
as well as $µ^*,\nu^* \in \mptd$ such that $W(t^*,µ^*,\nu^*) + \kappa t^* = 0 = \min_{µ,\nu} W(t^*,µ,\nu) + \kappa t^*$ (recall that $W$ is assumed to be smooth and $\mptd$ is compact). From the smoothness of $W$, we know that $\partial_t W(t^*,µ^*,\nu^*) \leq - \kappa$. Moreover, using the first order conditions recalled in Proposition \ref{firstcond}, and Hypothesis \ref{hypmon}, we deduce that
\begin{equation}\label{eq:intercomp}
\begin{aligned}
&\kappa \leq  -  \beta\left\langle \frac{\delta W}{\delta µ}(t^*,µ^*,\nu^*), \Delta µ^* \right\rangle - \beta\left\langle \frac{\delta W}{\delta \nu}(t^*,µ^*,\nu^*), \Delta \nu^* \right\rangle -\\
&-\beta   \left \langle \nabla µ^* \bigg| \frac{\delta^2 W}{\delta µ ^2}(t^*,µ^*,\nu^*,\cdot,\cdot) \bigg|\nabla µ^*\right\rangle +  \beta \left \langle \nabla \nu^* \bigg| \frac{\delta^2 W}{\delta \nu ^2}(t^*,µ^*,\nu^*,\cdot,\cdot) \bigg|\nabla \nu^*\right\rangle-\\
 &- 2 \beta\left \langle \nabla \nu^* \bigg| \frac{\delta^2 W}{\delta µ \delta \nu}(t^*,µ^*,\nu^*,\cdot,\cdot) \bigg|\nabla \mu^*\right\rangle.
\end{aligned}
\end{equation}
Let us now remark that by construction of $µ^*,\nu^*$, for any map $\mathcal{T} : \mptd^2 \to \mptd^2$
\begin{equation}\label{eq:intersign}
U(t^*,µ^*,\nu^*) - U(t^*,\mathcal{T}(µ^*,\nu^*)) \leq 0.
\end{equation} 
Hence, recalling the computation of section \ref{sec:translations} (which are here legit since $W$ is smooth), we deduce that the right hand side of \eqref{eq:intercomp} is non positive since it is the limit of a sum of terms of the form of the left hand side of \eqref{eq:intersign}, which are all non-positive. Thus we have in fact proven that $\kappa \leq 0$ which is a contradiction.\\

We deduce from the previous argument that $W$ is in fact non-negative. The rest of the proof follows exactly as the end of the proof of Theorem \ref{existclassic}.

\end{proof}

Observing the similarity between this equation and the first order one \eqref{me}, it is very tempting to formulate an adaptation of the definitions of monotone solutions we gave earlier. Two main difficulties arise at this point. The first one has to do with the fact that \eqref{me2} is of second order in the measure argument. As monotone solutions are reminiscent of viscosity solutions, similar technical difficulties naturally arise, we refer to \citep{crandall1992user} for more details on viscosity solutions of second order equations in a finite dimensional setting. The second difficulty has to do with the precise nature of \eqref{me2}. More or less, the definition of monotone solution we give consists in reformulating that $W : (t,µ) \to \langle U(t,µ) - \phi,µ - \nu\rangle$ is a sort of super solution of a certain PDE. The problem is here that in this second order case, the PDE satisfied by $W$ cannot be expressed solely in terms of $W$ and its derivatives, there is a term which involves directly the first order derivative of $U$ which cannot be rewritten with $W$. Indeed, if $U$ solves \eqref{me2} then $W$ defined above for $\phi \in \mathcal{C}^2, \nu \in \mptd$ satisfies
\begin{equation}\label{eqW2}
\begin{aligned}
\partial_t W& +  \langle - (\sigma + \beta) \Delta U + H(\cdot,\nabla_x U) - f(µ), µ - \nu \rangle - \left \langle \frac{\delta W}{\delta µ}, \text{div}(D_pH(\cdot,\nabla U)µ)\right \rangle\\
&+ \left \langle U - \phi, \text{div}(D_p H(\cdot, \nabla U)µ)\right\rangle - (\sigma + \beta)\left \langle \frac{\delta W }{\delta µ}, \Delta µ\right \rangle + (\sigma + \beta)\langle U - \phi, \Delta µ\rangle\\
&+ 2 \beta \left\langle\nabla \nu \bigg| \dmU (\cdot,µ,\cdot) \bigg| \nabla µ\right \rangle -\beta  \left \langle \nabla µ \bigg| \frac{\delta^2 W}{\delta µ ^2}(\cdot,\cdot) \bigg|\nabla µ\right\rangle = 0.
\end{aligned}
\end{equation}
A way to define monotone solutions could be to ask $W$ to be a super solution of this equation. If we do so, the previous proof is still valid, at least formally. The challenge here is that the second to last term cannot be written without using derivative (with respect to the measure variable) of $U$. This difficulty could be overcome by asking the couple $(U,W)$ to be a super solution of this equation. This idea is probably the most general one, for instance in the finite state space case, it is immediate to verify that it yields uniqueness of solutions by proving the non-negativity of $\langle U - V, µ - \nu\rangle$ for $U$ and $V$ two monotone solutions, using the techniques of \citep{crandall1992user}. However, because the technical difficulties usually arising from viscosity solutions of second order are much more difficult to treat in this infinite dimensional setting we focus on a regime which is more regular than the one we adopted up to this point. Formally we restrict ourselves to $\mathcal{C}^{1}$ solution in the measure argument and we use the limit \eqref{asymptotic2} to express the second order derivatives. This discussion leads us to the following definition.
\begin{Def}\label{def:me2}
A function $U$ in $\mathcal{B}_t$, which is derivable with respect to the measure variable at every point, is a monotone solution of \eqref{me2} if 
\begin{itemize}
\item for any $\mathcal{C}^2$ functions $\phi : \mathbb{T}^d \to \mathbb{R}, \vartheta : (0,\infty) \to \mathbb{R}$, for any measure $\nu \in \mmtd$, $T> 0$, and $(t_0,m_0) \in (0,T]\times \mptd$ point of strict minimum of the function $W$ defined by $(t,m) \to \langle U(t,\cdot,m)- \phi, m- \nu \rangle - \vartheta(t)$ the following holds
\begin{equation}
\begin{aligned}
\frac{d \vartheta}{dt}(t_0)& +  \langle - \sigma \Delta U + H(\cdot,\nabla_x U) , m_0 - \nu \rangle +2 \beta \left\langle\nabla \nu \bigg| \dmU (\cdot,µ,\cdot) \bigg| \nabla µ\right \rangle \\
 & + \beta \liminf_{h \to 0}\sum_{i =1 }^dh^{-2}(2W(t_0,m_0) - W(t_0,T_{+i}^hm_0)- W(t_0,T_{-i}^hm_0))\\
 \geq& \langle f(\cdot,m_0),m_0 - \nu\rangle-  \langle U - \phi, \text{div}(D_pH(\nabla_x U)m_0)\rangle  - \sigma \langle \Delta(U- \phi),m_0\rangle.
\end{aligned}
\end{equation}
\item The initial condition holds
\begin{equation}
U(0,\cdot,\cdot) = U_0(\cdot,\cdot).
\end{equation}
\end{itemize}
\end{Def}
In the previous definition, we used the notation
\begin{equation}
\forall 1 \leq i \leq d,m \in \mptd,  T_{\pm i}^h m= (Id \pm h e_i)_{\#}m,
\end{equation}
where for $1 \leq i \leq d$, $e_i = (0,...,1,...,0)$ and the sole $1$ is in position $i$. For this notion of solution, we can establish the following result of uniqueness

\begin{Theorem}
Under Hypothesis \ref{hypmon}, two monotone solutions of \eqref{me2} in the sense of Definition \ref{def:me2} only differ by a function $c : [0,\infty)\times \mptd \to \mathbb{R}$. If such a solution $U$ exists, then $U(t)$ is monotone for all time $t \geq 0$. 
\end{Theorem}
\begin{Rem}
Once again, a similar developments as in Section \ref{sec:extensionuniqueness} could be carried on here.
\end{Rem}
\begin{proof}
Let us consider $U$ and $V$ two solutions, and define $W$ by
\begin{equation}
W(t,s,µ,\nu) = \langle U(t,\cdot,\mu) - V (s,\cdot,\nu), µ - \nu \rangle.
\end{equation}
We want to show that $W(t,t,\cdot,\cdot) \geq 0$ for all $t\geq 0$. Once again we argue by contradiction. Hence, we assume that there exists $t_*,\delta, \bar{\epsilon} > 0$, such that for all $\epsilon \in (0,\bar{\epsilon}), \alpha > 0, \phi, \psi \in \mathcal{C}^2$ such that $\|\phi\|_2 + \| \psi \|_2 \leq \epsilon$ and $\gamma_1,\gamma_2 \in (\frac{\bar{\epsilon}}{2}, \bar{\epsilon})$, 
\begin{equation}\label{infuniqt2}
\inf_{t,s \in [0,t_*], µ,\nu \in \mptd} \left\{ W(t,s,µ,\nu) + \langle \phi, µ\rangle + \langle \psi, \nu \rangle + \frac{1}{2 \alpha}(t-s)^2 + \gamma_1 t + \gamma_2 s \right\} \leq - \delta.
\end{equation}
From Lemma \ref{Stegall}, we know that there exists (for any value of $\alpha$) $\phi, \psi, \gamma_1$ and $\gamma_2$ such that $(t,s,µ,\nu) \to W(t,s,µ,\nu) + \langle \phi, µ\rangle + \langle \psi, \nu \rangle + \frac{1}{2 \alpha}(t-s)^2 + \gamma_1 t + \gamma_2 s$ has a strict minimum on $[0,t_*]^2 \times \mptd^2$ at $(t_0,s_0,\mu_0,s_0)$.\\

We now assume that $t_0, s_0 > 0$ (the case in which one of the two is zero is treated exactly as in Theorem \ref{uniqtime}).
Let us now remark the following 
\begin{equation}
\begin{aligned}
0 \geq& \liminf_{h \to 0}\sum_{i =1 }^dh^{-2}(2W(t_0,s_0,µ_0,\nu_0) - W(t_0,s_0,T_{+i}^hµ_0,T_{+i}^h\nu_0)- W(t_0,w_0,T_{-i}^hµ_0,T_{-i}^h\nu_0))\\
= &\liminf_{h \to 0}\sum_{i =1 }^dh^{-2}(2W(t_0,s_0,µ_0,\nu_0)- W(t_0,s_0,T_{+i}^hµ_0,\nu_0)- W(t_0,w_0,T_{-i}^hµ_0,\nu_0))\\
&+ \liminf_{h \to 0}\sum_{i =1 }^dh^{-2}(2W(t_0,s_0,µ_0,\nu_0)- W(t_0,s_0,µ_0,T_{+i}^h\nu_0)- W(t_0,w_0,µ_0,T_{-i}^h\nu_0))\\
&+ 2 \beta \left(\left\langle\nabla \nu_0 \bigg| \dmU (\cdot,µ_0,\cdot) \bigg| \nabla µ_0\right \rangle + \left\langle\nabla \mu_0 \bigg| \frac{\delta V}{\delta m} (\cdot,\nu_0,\cdot) \bigg| \nabla \nu_0\right \rangle \right)
\end{aligned}
\end{equation}
Using the definition of monotone solutions of both $U$ and $V$, we then arrive at a contradiction following the same argument as in the proof of Theorem \ref{uniqtime}.
\end{proof}
\begin{Rem}
Let us insist on the fact that for two smooth functions $U,V: \mathbb{T}^d\times \mptd \to \mathbb{R}$, the crossed derivative of $W (µ,\nu) = \langle U(µ) - V(\nu),µ - \nu\rangle$ is given, up to a constant, by
\begin{equation}
\forall x,y \in \mathbb{T}^d, \frac{\delta^2 W}{\delta µ \delta \nu}(µ,\nu,x,y) = \dmU(µ,y,x) + \frac{\delta V}{\delta m}(\nu,x,y).
\end{equation}
\end{Rem}
\begin{Rem}
If we were to work on an Hilbert space, then an analogous of Theorem 3.2 in \citep{crandall1992user} could probably have been established using the properties of semi convex functions in Hilbert space such as in \citep{lasry1986remark}, which would have allowed us to use a weaker notion of solutions.
\end{Rem}
Let us remark that following Section \ref{sec:exist}, we can obtain a result of existence of monotone solutions of \eqref{me2} by weakening the assumptions of Theorem 2.11 in \citep{cardaliaguet2019master}. As the approach is very similar to the one of Section \ref{sec:exist}, we only state the following without a demonstration.
\begin{Theorem}\label{thm:exist2}
Assume that the assumptions of Theorem \ref{existclassic} hold and that $f$ and $U_0$ can be approximated uniformly by functions $f_n$ and $U_{0,n}$ satisfying for some $C_n>0$
\begin{equation}
\begin{aligned}
\sup_{m \in \mptd}& \left(   \|f(\cdot,m)\|_{2 + \alpha} + \left\| \frac{\delta f(\cdot,m,\cdot)}{\delta m}  \right\|_{(2+ \alpha, 2 + \alpha)} + \left\| \frac{\delta^2 f(\cdot,m,\cdot,\cdot)}{\delta m^2}  \right\|_{(2+ \alpha, 2 + \alpha,2 + \alpha)} \right)\\
 &+ \text{Lip}_{2 +\alpha} \left( \frac{\delta^2 f}{\delta m^2} \right) \leq C_n.
 \end{aligned}
\end{equation}
 \begin{equation}
 \begin{aligned}
\sup_{m \in \mptd}& \left(   \|U_0(\cdot,m)\|_{3 + \alpha} + \left\| \frac{\delta U_0(\cdot,m,\cdot)}{\delta m}  \right\|_{(3+ \alpha, 3 + \alpha)} + \left\| \frac{\delta^2 U_0(\cdot,m,\cdot,\cdot)}{\delta m^2}  \right\|_{(3+ \alpha, 3 + \alpha, 3 + \alpha)} \right)\\
& + \text{Lip}_{3 +\alpha} \left( \frac{\delta^2 U_0}{\delta m^2} \right) \leq C_n.
\end{aligned}
\end{equation}
Then there exists a monotone solution of \ref{me2} in the sense of Definition \ref{def:me2}.
\end{Theorem}
\begin{Rem}
As the second order terms in \eqref{me2} can be seen as limits of common jumps for jump operators which are non expansive in $L^2(\mathbb{T}^d)$ (they are translations), let us remark that the a priori estimate established in Proposition \ref{prop:est} is also valid for \eqref{me2}. This could have been observed by applying directly the same technique as in the proof of Proposition \ref{prop:est} on \eqref{me2}.
\end{Rem}

\section{Conclusion and future perspectives}\label{sec:conclusion}
We have presented a notion of solution for MFG master equations which allows us to work with solutions which are merely continuous for first order equations and only one time differentiable for second order equations (each time with respect to the measure argument). This notion is built to enjoy nice uniqueness properties but it also verifies strong stability properties. Even though we do not treat exhaustively the question of existence, the stability properties is of course helpful to establish such results. Let us mention that the generalization of our results to master equation\color{black}s \color{black}which do not model exactly MFG but which have a monotone structure (such as in \citep{lions2020extended} in the continuous setting) is of course immediate following \citep{bertucci2020monotone} modulo some assumptions on the non-linearities. Another immediate extension which follows from \citep{bertucci2020monotone} is the one to master equations having another type of monotonicity than the usual one, such as in section 2.3 in \citep{bertucci2020monotone}. The extension to master equations on $\mathcal{M}_+(\mathbb{T}^d) := \{m \in \mmtd| m \geq 0\}$ is also straightforward.\\

At this point, we believe worth mentioning several extensions or future directions to explore, that we believe to be meaningful. Perhaps the most important one concerns the extension of monotone solutions of second order to merely continuous functions of the measure argument. In this setting, such an improvement should probably pass by a better understanding of the derivatives of functions on $\mptd$, in order to establish an analogous of Theorem 3.2 in \citep{crandall1992user}. Another interesting extension is the adaptation of the previous results to MFG of optimal stopping or other singular controls, as it has been done in \citep{bertucci2020monotone} in the finite state space case. For such problems, a uniqueness result for monotone solutions of the master equation (which is now posed on the set of non-negative measure and not on the set of probability measures) follows quite immediately from the the present study, using a similar formulation such as in \citep{bertucci2020monotone}. However, existence questions require key estimates which have not been proven at the moment.\\

Other extensions of this work include the study of problems on more general domain than the Torus, thus involving boundary conditions, as well as numerical methods to approximate monotone solutions.

\bibliographystyle{plainnat}
\bibliography{bibremarks}

\appendix

\end{document}